\documentclass[11pt,reqno]{amsart}
\usepackage{amssymb,amsmath}
\usepackage{amsthm}
\usepackage{color,graphicx}
\usepackage{hyperref}
\usepackage{color}
\usepackage{graphicx}
\usepackage{epstopdf}
\usepackage{subfigure}

\newcommand{\R}{\mathbb{R}}
\newcommand{\N}{{\mathbb N}}

\newcommand{\Z}{\mathbb Z}

\newcommand{\mtm}{\mathcal M}
\newcommand{\Lk}{\mathcal{L}_k}

\newcommand{\p}{\partial}

\newcommand{\ds}{\displaystyle}
\newcommand{\sn}{{\mbox{SN}}}
\newcommand{\cn}{{\mbox{CN}}}
\newcommand{\dn}{{\mbox{DN}}}

\numberwithin{equation}{section}

\newtheorem{theorem}{Theorem}[section]

\newtheorem{remark}[theorem]{Remark}
\newtheorem{lemma}[theorem]{Lemma}

\newtheorem{definition}[theorem]{Definition}

\begin{document}
\vglue-1cm \hskip1cm
\title[Stability of periodic waves for dispersive equations]{Orbital stability of one-parameter periodic traveling waves for dispersive equations and applications}



\begin{center}


\subjclass[2010]{Primary 35A01, 35Q53 ; Secondary 35Q35}

\keywords{Dispersive equations; Orbital stability; Periodic waves}

\maketitle

{\bf Thiago Pinguello de Andrade}

{Departamento de Matem{\'a}tica - Universidade Tecnol{\'o}gica Federal do Paran{\'a}  \\
Av. Alberto Carazzai, 1640, CEP 86300-000, Corn{\'e}lio Proc{\'o}pio, PR, Brazil.}\\
{thiagoandrade@utfpr.edu.br}

\vspace{3mm}

{\bf Ademir Pastor}

{ IMECC-UNICAMP\\
Rua S\'ergio Buarque de Holanda, 651, CEP 13083-859, Campinas, SP,
Brazil.  } \\{ apastor@ime.unicamp.br}

\end{center}

\begin{abstract}
This paper establishes sufficient conditions for the orbital stability of one-parameter spatially periodic traveling-wave solutions for one-dimensional dispersive equations. Our method of proof combines known techniques with some new ideas. As a consequence of our result, we give several applications for well known dispersive equations. Extension of the theory to regularized equations is also established.
\end{abstract}

\section{Introduction}

 This paper sheds new contributions on the orbital stability theory  of one-parameter periodic traveling-wave solutions for nonlinear  dispersive models which can be written in the forms
\begin{equation}\label{gkdv}
u_t-\mtm u_x+\partial_x(f(u))=0
\end{equation}
and 
\begin{equation}\label{gbbm}
u_t+\mtm u_t+\partial_x(u+f(u))=0,
\end{equation}
where $f:\R\to\R$ is  a (at least) $C^1$-function, in general representing the nonlinearity, and $\mtm$ is a differential or  pseudo-differential operator.

Equations as in \eqref{gkdv} and \eqref{gbbm} appear in many physical situations. For instance, it describes long-crested, long-wavelength disturbances of small
amplitude propagating primarily in one direction in a dispersive media (see \cite{bbm}). In particular, when $\mathcal{M}=-\partial_x^2$ and $f$ is a function having the form $f(u)=u^{p+1}$, where $p\geq1$ is an integer number, equation \eqref{gkdv} is the well known generalized Korteweg-de Vries (KdV) equation
\begin{equation}
u_t+u_{xxx}+\partial_x(u^{p+1})=0,
\end{equation} 
whereas \eqref{gbbm} reduces to the generalized regularized long-wave equation or generalized Benjamin-Bona-Mahony (BBM) equation 
\begin{equation}
u_t-u_{xxt}+u_x+\partial_x(u^{p+1})=0.
\end{equation} 

We point out that we will be primarily interested in equations having the form \eqref{gkdv}, because the theory can be easily extended to \eqref{gbbm} (see Section \ref{gbbmsec}).
Traveling-wave solutions (or,  for short, traveling waves) for \eqref{gkdv}  are those solutions having the form 
\begin{equation}\label{travform}
u(x,t)=\phi(x-ct),
\end{equation}
where $\phi:\R\to\R$ is a suitable function representing the profile of the wave and $c$ is a real constant representing the wave speed. The traveling waves of main interest are classified into two classes: \textit{solitary} and \textit{periodic} traveling waves. Solitary waves are those for which $\phi$, together with all its derivatives go to zero at infinity, whereas periodic traveling waves are those for which $\phi$ is a periodic function of its argument with a fixed period $L>0$.  The study of existence and stability (linear and nonlinear) of traveling waves has gained a lot of attention is recent decades. The interested reader will find a vast number of works in the current literature, which we refrain from list them here. Instead, we will focus in the works closed to ours.

 To obtain traveling waves, we usually replace the ansatz \eqref{travform} into \eqref{gkdv} and try to determine $\phi$ and $c$ which provide the desired solutions. 
Thus, by replacing this waveform in \eqref{gkdv}, we see that $\phi$ must be a solution of the differential or pseudo-differential equation
\begin{equation}\label{soleq}
(\mtm +c)\phi-f(\phi)+A=0,
\end{equation}
where $A$ appears as an integration constant. Thus, when we are interested in the study of traveling waves, solving  equation \eqref{soleq}, which turns out to be a  two-parameter equation, is the first step. There are several manner of finding traveling waves and the most popular ones, to cite a fews, are the quadrature method, by using the implicit function theorem or the Lyapunov-Schmidt method, applications of the critical point theory, and concentration-compactness techniques.

Let us now describe the framework in order to present our main result. Operator $\mathcal{M}$ shall be formally defined through its Fourier transform by
\begin{equation}\label{defM}
\widehat{\mtm u}(m)=\alpha(m)\widehat{u}(m), \qquad m\in \Z.
\end{equation}
Here and in what follows $\widehat{u}$ denotes the Fourier transform of the periodic function $u$.
The symbol $\alpha$ is assumed to be a measurable, locally bounded, even, and real  function satisfying the following:
\begin{itemize}
\item[(i)] there are real constants $s_1,s_2>0$  with $s_1\leq s_2$, and  $A_1, A_2>0$ such that
\begin{equation}\label{alphacond}
A_1|m|^{s_1}\leq\alpha(m)\leq A_2(1+|m|)^{s_2},
\end{equation}
for all $m$ sufficiently large;
\item[(ii)] there is a real constant $\gamma$ such that $\inf_{m\in \Z}\alpha(m)\geq \gamma$.
\end{itemize}

 In many applications the parameters $c$ and $A$ appearing in \eqref{soleq} are not independent themselves and it turn out to depend on a third parameter, which we shall denote by $k$. This is the case, for instance, when $\mtm$ is a differential operator and $f$ is a power-law function: in several situations the solutions of \eqref{soleq} depend on the well-known Jacobian elliptic functions and, as a consequence, the parameters $c$ and $A$ depend on the elliptic modulus $k\in(0,1)$. Having this in mind, our first assumption is the following one.
\begin{itemize}
\item[{\bf (H0)}] There are an interval $J\subset\R$, $C^1$-functions $k\in J\mapsto c=c(k)$ and $k\in J\mapsto A=A(k)$, and a nontrivial smooth curve of $L$-periodic solutions for \eqref{soleq}, say, $ k\in J  \mapsto \phi_k:=\phi_{(c(k),A(k))} \in H^{s_2}_{per}([0,L])$ with $c=c(k)>-\gamma$.
\end{itemize}
Here and in what follows $H^s_{per}([0,L])$ will denote the periodic Sobolev space of order $s$ (see definition below). The condition $c>-\gamma$ is necessary in order to make the operator $\mathcal{M}+c$ positive.

After the linearization of \eqref{gkdv} around a periodic solution $\phi=\phi_k$, we are faced an unbounded, closed, self-adjoint operator $\mathcal{L}_k:D(\mathcal{L}_k)\subset L^2_{per}([0,L])\to L^2_{per}([0,L])$, defined on a dense subspace, by
\begin{equation}\label{defL}
\mathcal{L}_ku:=(\mtm +c)u-f'(\phi_k)u.
\end{equation}
 Our assumptions concerning  $\mathcal{L}_k$ are the following.

\begin{itemize}
\item[{\bf(H1)}] The linearized operator $\mathcal{L}_k:=\mathcal{L}_{(c(k),A(k))}$ has a unique negative  eigenvalue, which is simple.
\item[{\bf(H2)}] Zero is a simple eigenvalue of $\mathcal{L}_k$ with associated  eigenfunction $\phi_k'$.
\end{itemize}

In order to introduce the remaining assumptions, let us recall that, at least formally, \eqref{gkdv} conserves the quantities
\begin{equation}\label{energy}
E(u)=\frac{1}{2}\int_0^L (u\mtm u-2F(u))dx, \quad \mbox{with} \quad F(u)=\int_0^uf(s)ds,
\end{equation}
\begin{equation}\label{mass}
Q(u)=\frac{1}{2}\int_0^L u^2dx,
\end{equation}
and
\begin{equation}\label{mm}
V(u)=\int_0^L u\,dx.
\end{equation}
Note that \eqref{soleq} is nothing but the Euler-Lagrange equation associated with the functional  
\begin{equation}\label{Fkfunctional}
F_k:=E+cQ+AV,
\end{equation}
that is, solutions of \eqref{soleq} are critical points of $F_k$. Thus, as is well understood, the nature of these points are crucial to determine their stability. In addition,  $\mathcal{L}_k$ is nothing but the second order Fr\'echet derivative  of $F_k$ at $\phi_k$, that is, $\mathcal{L}_k=F_k''(\phi_k)$. Thus, it is expected that the spectrum of $\mathcal{L}_k$ plays a crucial role in the stability analysis. We will show that {\bf (H1)-(H2)} is sufficient to our purposes.

Next, for a fixed $k\in J$, we introduce the  functional
$$
M_k(u):=\ds\frac{\partial c}{\partial k}Q(u)+\frac{ \partial A}{\partial k} V(u),
$$
where $\frac{\partial c}{\partial k}$ and $\frac{\partial A}{\partial k}$ denote, respectively, the derivatives of the functions $c(k)$ and $A(k)$ at $k$.
We assume the following.
\begin{itemize}
\item[{\bf(H3)}] The quantity $\Phi$ defined by $\Phi:=\left\langle\mathcal{ L}_k\left(\frac{\partial \phi_k}{\partial k}\right),\frac{\partial \phi_k}{\partial k}\right\rangle$
 is negative.
 \item[{\bf(H4)}] It holds $M_k(\phi_k)\neq -\dfrac{\partial c}{\partial k}Q(\phi_k)$.
\end{itemize}

Once property {\bf (H0)} has been proved, our main goal is to show that {\bf (H1)-(H4)} combine to establish the orbital stability of the traveling wave $\phi_k$, for each $k\in J$ fixed.
In order to make clear the definition of orbital stability, let us observe that \eqref{alphacond} implies that the operator $\mathcal{L}_k$ is well-defined on $H^{s_2}_{per}([0,L])$ and the natural Sobolev space to  consider the flow of \eqref{gkdv} is the energy space $H^{s_2/2}_{per}([0,L])$.

\begin{definition}
Let $\phi_k$ be an $L$-periodic solution of \eqref{soleq}. We say that $\phi_k$ is orbitally stable (by the flow of \eqref{gkdv}) in  $H^{s_2/2}_{per}([0,L])$ if, for any $\varepsilon>0$, there exists $\delta>0$ such that if $u_0\in H^{s_2/2}_{per}([0,L])$ satisfies
$$
\|u_0-\phi\|_{H^{s_2/2}_{per}}<\delta,
$$
then the solution $u(t)$ of \eqref{gkdv}, with initial data $u_0$, exists globally and satisfies
$$
\sup_{t\in\R}\inf_{r\in\R}\|u(t)-\phi_k(\cdot+r)\|_{H^{s_2/2}_{per}}<\varepsilon.
$$
Otherwise, we say that $\phi_k$ is $H^{s_2/2}_{per}$-unstable.
\end{definition}

\begin{remark}
Note that, by definition, if the Cauchy problem associated with \eqref{gkdv} is not globally well-posed in $H^{s_2/2}_{per}([0,L])$, at least for initial data in a small neighborhood of $\phi_k$,  then any traveling wave $\phi_k$ is $H^{s_2/2}_{per}$-unstable. Since the issue of well-posedness is out of the scope of this manuscript, in what follows we will assume that \eqref{gkdv} always admit global solutions for initial data in $H^{s_2/2}_{per}([0,L])$, that is, for any $u_0\in H^{s_2/2}_{per}([0,L])$, \eqref{gkdv} has a unique solution $u$ satisfying $u(0)=u_0$ and $u\in C([-T,T]; H^{s_2/2}_{per}([0,L]))$, for any $T>0$.
\end{remark}

Our main theorem concerning orbital stability reads as follows.

\begin{theorem}[Orbital stability]\label{gkdvmaintheorem}
Under assumptions {\bf (H0)-(H4)},
for each $k\in J$, the periodic traveling wave $\phi_k$
 is orbitally stable  in $H^{s_2/2}_{per}([0,L])$.
\end{theorem}

The strategy to prove Theorem \ref{gkdvmaintheorem} follows the classical arguments in   \cite{bss} and \cite{Grillakis}. Roughly speaking, if one restricts the augmented energy functional to a suitable manifold (see \eqref{Mkmanifold}) then the critical point in question is a minimum, which in turn implies the orbital stability.

Before proceeding, a few words of explanation concerning assumptions {\bf (H3)-(H4)} are in order. Assume for the moment that $A=0$ and \eqref{soleq} has a family of solutions $c\mapsto \phi_c$, for $c$ in an open interval. In this case, as is well-known in the current literature (see, for instance, \cite{bss} and \cite{Grillakis}), the Vakhitov-Kolokolov type condition
\begin{equation}\label{VKcond}
\dfrac{d}{d c}\int \phi_c^2\ dx>0,
\end{equation}
together with assumptions {\bf (H1)-(H2)} (with $\mathcal{L}_k$ replaced by $\mathcal{L}_c=\mathcal{M}+c-f'(\phi_c)$) is sufficient to imply the stability of $\phi_c$. Taking the derivative with respect to $c$ in \eqref{soleq} we deduce that $\mathcal{L}_c(\partial_c \phi_c)=-\phi_c$. Therefore,
$$
\dfrac{1}{2}\dfrac{d}{dc}\int \phi_c^2\ dx=\int \phi_c \partial_c \phi_c\ dx=-\int \mathcal{L}_c(\partial_c \phi_c)\partial_c \phi_c\ dx=-\langle \mathcal{L}_c(\partial_c \phi_c), \partial_c \phi_c \rangle,
$$
and \eqref{VKcond} is equivalent to
\begin{equation}\label{VKcond1}
\langle \mathcal{L}_c(\partial_c \phi_c), \partial_c \phi_c \rangle<0.
\end{equation}
Thus, condition $\Phi<0$ can be viewed as a generalization of \eqref{VKcond1} in the context of the present manuscript. Also, in the situation described in this paragraph, for which $c=k$, we  have that {\bf (H4)} is fulfilled provide $\phi_c\neq0$. Thus, {\bf (H4)} appears as an extra assumption in our context. It should be noted that {\bf (H4)} is indeed used to prevent $M_k(\phi_k)$ from being a critical value of a parabola defined in the proof of Theorem \ref{gkdvmaintheorem}. We believe, however, that this assumption is not restrictive as is shown in our applications.

Next, let us try to relate our work with the ones in the current literature (see also our applications below). The stability of spatially periodic traveling waves for dispersive equations was initiated by T.B. Benjamin (see \cite{be} and \cite{be1}), when studying the stability of cnoidal waves of the KdV equation. It should pointed out, however, that the orbital stability of such waves was completed a couple of decade later (see \cite{angulo1}). Specially after \cite{angulo1}, much effort has been expended on the stability theory of periodic traveling waves, and the issue has been attracted the attention of a much broader community of mathematicians and physicists.

Most of the works in the literature do not assume that $c$ and $A$ depend on a third parameter and, instead, solutions of \eqref{soleq} are parametrized by $c$ and $A$ themselves (or even more parameters). Specially in the case $\mathcal{M}=-\partial_x^2$,  the main results in this direction are provided by M. Johnson \cite{johnson1} and collaborators (see also \cite{bjk}, \cite{johnson2}, and references therein). Equation \eqref{soleq} are then written as
\begin{equation}\label{kdvsoleq}
-\phi''+c\phi-f(\phi)+A=0.
\end{equation}
Multiplying \eqref{kdvsoleq} by $\phi'$ and integrating once, we obtain
\begin{equation}\label{kdvsoleq1}
-\frac{1}{2}(\phi')^2+\frac{c}{2}\phi^2-F(\phi)+A\phi=B,
\end{equation}
where $B$ in another integration constant, interpreted as the energy. If the effective potential
$$
\Gamma(\phi)=-\frac{c}{2}\phi^2+F(\phi)-A\phi
$$
has a local minimum, then from the theory of differential equations, \eqref{kdvsoleq1} has periodic solutions which can be parametrized by the parameters $(c,A,B)$. In \cite{johnson1}, \cite{bjk} the author then establishes some criteria to determine the orbital stability of such waves depending on the sign of certain determinants, which encode some geometric information about the underlying manifold of periodic solutions. This approach is much  general. Indeed, it brings many important contributions to the theory of orbital stability for periodic traveling waves and it has been successfully applied in several situations. For one  hand, our approach is more particular and it does not provide some criterion to study the orbital stability of \textit{all} periodic solution of \eqref{soleq} (or \eqref{kdvsoleq}). On the other hand,  when applicable, our method yields some simplifications and even provides new results.

Besides this introduction, the paper is organized as follows. In Section \ref{section2} we prove Theorem \ref{gkdvmaintheorem}. In Section \ref{applic} we give the applications of our method to the Korteweg-de Vries, modified Korteweg-de Vries, Gardner, Intermediate Long Wave, and Schamel equations. Extension to regularized equations as in \eqref{gbbm} will be given in Section \ref{gbbmsec}. Applications to the regularized Schamel and modified Benjamin-Bona-Mahony equations are also provided.\\

\noindent {\bf Notation.} For $s\in\mathbb{R}$, the Sobolev space
$H_{per}^{s}=H_{per}^{s}([0,L])$ is the set of all periodic distributions
such that $||f||_{H^s_{per}}^2:=
L\sum_{k=-\infty}^{+\infty}(1+|k|^2)^s|\widehat{f}(k)|^2 <\infty, $ where
$\widehat{f}$ is the (periodic) Fourier transform of $f$. For $s=0$, $H_{per}^{s}([0,L])$ is isometric to $L_{per}^{2}([0,L])$. The norm and inner product  in $L^2_{per}$ will be denoted by $\|\cdot\|$ and $(\cdot,\cdot)_{L^2_{per}}$, respectively. By $\langle\cdot,\cdot\rangle$ we mean the duality pairing $H^s_{per}$-$H^{-s}_{per}$.  The symbols $\sn(\cdot,k)$,
$\dn(\cdot,k)$, and $\cn(\cdot,k)$ represent the Jacobi elliptic functions of
\emph{snoidal}, \emph{dnoidal}, and \emph{cnoidal} type, respectively. Recall
that $\sn(\cdot,k)$ is an odd function, while $\dn(\cdot,k)$, and
$\cn(\cdot,k)$ are even functions. For $k\in(0,1)$, $K(k)$ and $E(k)$ will
denote the complete elliptic integrals of the first and second type,
respectively (see e.g., \cite{friedman}).

\section{Proof of Theorem \ref{gkdvmaintheorem}}\label{section2}

Our goal in this section is to prove Theorem \ref{gkdvmaintheorem}, that is, to show how Assumptions {\bf (H1)-(H4)} implies the orbital stability of the periodic traveling wave $\phi_k$.

To begin with, in $H^{s_2/2}_{per}([0,L])$, let us introduce the pseudo-metric $\rho$  defined by
\begin{equation}\label{rhodef}
\rho(u,v):=\inf_{r\in\R} \| u-v(\cdot+r)\|_{H_{per}^{s_2/2}}.
\end{equation}
Given any real number $\varepsilon>0$, $U_\varepsilon(\phi_k)$ shall denote the $\varepsilon$-neighborhood of $\phi_k$ with respect to $\rho$, that is,
$$
U_\varepsilon(\phi_k)=\{u\in H^{s_2/2}_{per}([0,L]); \; \rho(u,\phi_k)<\varepsilon\}.
$$
Also, we introduce the manifold $\Sigma_k$  as
\begin{equation}\label{Mkmanifold}
\Sigma_k:=\{u\in H^{s_2/2}_{per}([0,L]);\,\,  M_k(u)=M_k(\phi_k)\}.
\end{equation}

Now, we state two classical lemmas. The proofs in our case are very close to the original ones.

\begin{lemma}\label{bbm14}
There exist $\varepsilon>0$ and a $C^1$ map  $\omega:U_\varepsilon(\phi_k) \to \mathbb{R}$, such that for all $u\in U_\varepsilon(\phi_k)$,
$$\Big( u( \cdot+\omega(u)),\phi_k'\Big)_{L^2_{per}}=0.
$$
\end{lemma}
\begin{proof}
The proof is based on an application of Implicit Function Theorem. See \cite[Lemma 4.1]{bss} or \cite[Lemma 7.7]{angulo4} for details.
\end{proof}

\begin{lemma}\label{bbm15}
Let
$$
\mathcal{A}=\left\{ \psi\in H^{s_2/2}_{per}([0,L]); ( \psi,M_k'(\phi_k))_{L^2_{per}}= ( \psi, \phi_k')_{L^2_{per}}=0\right\}.
$$
Under assumptions  ${\bf (H0)}$-${\bf (H3)}$
 there exists $C>0$ such that 
$$\langle \mathcal{L}_k\psi,\psi\rangle \geq C \| \psi\|_{H^{s_2/2}_{per}}^2, \quad \psi \in \mathcal{A}.
$$
\end{lemma}
\begin{proof}
The proof is quite standard by now, so we omit the details. We refer the interested reader to \cite[Theorem 3.3]{Grillakis} or \cite[Lemma 7.8]{angulo4}. It should be noted that the assumption $d''(c)>0$ in such references must be replaced by $\Phi<0$ in our case.
\end{proof}

In the next lemma we prove that $\phi_k$ is a local minimum of the functional  $F_k$  restrict to the manifold $\Sigma_k$. It worth mentioning that its proof relies on the classical ideas with some changes in the spirit of  \cite[Lemma 4.6]{johnson1}.

\begin{lemma} \label{lemacoercividadebbm}
 Under the above assumptions, there exist $\varepsilon >0$  and a constant $C=C(\varepsilon)$ such that
\begin{equation}F_k(u)-F_k(\phi_k)\geq C\rho(u,\phi_k)^2,\label{bbm3}\end{equation}
for all  $u\in U_\varepsilon(\phi_k)$  satisfying  $M_k(u)=M_k(\phi_k)$.
\end{lemma}
\begin{proof}
Since $F_k$ is invariant by translations, we have $F_k(u)=F_k(u(\cdot+r))$, for all  $ r\in\mathbb{R}$.  Thus, it suffices to prove that
$$F_k(u(\cdot+\omega(u)))-F_k(\phi_k)\geq C \rho(u,\phi_k)^2,$$
where $\omega$ is given in Lemma \ref{bbm14}. 

By fixing $u\in U_\varepsilon(\phi_k)\cap\Sigma_k$ (with $\varepsilon$ as in Lemma \ref{bbm14}) and making use of Lemma \ref{bbm14}, it follows that there exists $C_1\in\mathbb{R}$ such that
$$v:=u(\cdot+\omega(u))-\phi_k=C_1M_k'(\phi_k)+y,$$
where $y\in \mathcal{T}_k := \{M_k'(\phi_k)\}^\perp \cap\{\phi_k'\}^\perp$. Since  $u$ belongs to $ U_\varepsilon(\phi_k)$, up to a translation  in $\phi_k$, we may assume that $v=u(\cdot+\omega(u))-\phi_k$ satisfies $\|v\|_{H^{s_2/2}_{per}}<\varepsilon$. 

Let us prove that $C_1=O(\|v\|^2)$. In fact, using the invariance by translation of  $M_k$, a Taylor expansion gives
\begin{equation}
M_k(u)=M_k(u(\cdot+\omega(u)))=M_k(\phi_k)+\langle M_k'(\phi_k),v\rangle +O(\|v\|^2).
\label{bbm5}
\end{equation}
On the other hand, since $y\in\mathcal{T}_k$, we have $\langle M_k'(\phi_k),y\rangle=0$ and
\begin{equation}\langle M_k'(\phi_k),v\rangle=\langle M_k'(\phi_k), C_1M_k'(\phi_k)+y\rangle=C_1\langle M_k'(\phi_k), M'_k(\phi_k)\rangle= C_1N,
\label{bbm6}
\end{equation}
where $N$ is a constant depending only on $k$. Therefore, since $M_k(u)=M_k(\phi_k)$, it follows from  (\ref{bbm5}) and (\ref{bbm6}) that
\begin{equation}
C_1=O(\|v\|^2).
\label{bbm7}
\end{equation}

A Taylor expansion  at   $u(\cdot +\omega(u))=\phi_k+v$ now yields
$$F_k(u)=F_k(u(\cdot+\omega(u)))=F_k(\phi_k)+\langle F_k'(\phi_k),v\rangle+\frac{1}{2}\langle F_k''(\phi_k)v,v\rangle + o(\|v\|^2).$$
Since $F'_k(\phi_k)=0$ and $F_k''(\phi_k)=\mathcal{L}_k $, we have 
\begin{equation}
F_k(u)-F_k(\phi_k)=\frac{1}{2}\langle \mathcal{L}_k  v,v\rangle+o(\|v\|^2).
\label{bbm8}
\end{equation}
Note that the equality $v=C_1M_k'(\phi_k)+y$ provides
\begin{equation}\label{bbm8.1}
\langle \mathcal{L}_k  v,v\rangle  = C_1^2 \langle\mathcal{L}_k  M'_k(\phi_k),M'_k(\phi_k)\rangle+ 2 C_1\langle \mathcal{L}_k  M'_k(\phi_k), y\rangle +\langle\mathcal{L}_k  y,y\rangle.
\end{equation}
In view of \eqref{bbm7}, we obtain positive constants $C_2, C_3$, and $C_4$, depending only on $k$, such that
$$|C_1^2 \langle\mathcal{L}_k  M'_k(\phi_k),M'_k(\phi_k)\rangle|\leq C_2\|v\|^4$$
and 
\[
\begin{split}
\ds |2C_1\langle \mathcal{L}_k  M'_k(\phi_k), y\rangle|& \ds\leq 2 |C_1| \|\mathcal{L}_k  M'_k(\phi_k)\|\|y\|\\
&\ds \leq 2 |C_1| \|\mathcal{L}_k  M'_k(\phi_k)\|\Big(\|y+C_1M_k'(\phi_k)\|+\|C_1M_k'(\phi_k)\|\Big)\\
&\ds\leq C_3\|v\|^3+C_4\|v\|^4.\\
\end{split}
\]
This last two inequalities together with \eqref{bbm8.1} imply
\begin{equation}\langle \mathcal{L}_k  v,v\rangle=\langle\mathcal{L}_k  y,y\rangle+ o(\|v\|^2).
\label{bbm9}
\end{equation}
Therefore, combining (\ref{bbm8}) with (\ref{bbm9}), we obtain
$$F_k(u)-F_k(\phi_k)=\frac{1}{2}\langle \mathcal{L}_k  y,y\rangle+o(\|v\|^2).$$
By using that $y\in\mathcal{T}_k$, we have  $y\in\mathcal{A}$. Thus, Lemma \ref{bbm15} gives
$$\langle \mathcal{L}_k  y,y\rangle \geq C \| y\|_{H^{s_2/2}_{per}}^2,$$ 
and, consequently,
\begin{equation}
F_k(u)-F_k(\phi_k)\geq C\|y\|_{H^{s_2/2}_{per}}^2+o(\|v\|^2).
\label{bbm12}
\end{equation}
By using the definition of $v$ and \eqref{bbm7} it is easily seen that 
\begin{equation}
\|y\|_{H^{s_2/2}_{per}}^2\geq \|v\|_{H^{s_2/2}_{per}}^2+o(\|v\|_{H^{s_2/2}_{per}}^2),
\label{bbm13}
\end{equation}
provided $v$ is small enough (if necessary we can take a smaller $\varepsilon>0$).

Finally,  (\ref{bbm12}) and (\ref{bbm13}) combine to establish that
$$
F_k(u)-F_k(\phi_k)\geq C\|v\|_{H^{s_2/2}_{per}}^2+o(\|v\|_{H^{s_2/2}_{per}}^2),
$$
which, for $\varepsilon>0$ sufficient small, gives
$$
F_k(u)-F_k(\phi_k)\geq C(\varepsilon)\|v\|_{H^{s_2/2}_{per}}^2\geq C(\varepsilon) \rho(u,\phi_k)^2
$$
and completes the proof of the lemma.
\end{proof}

Finally, we are in a position to prove Theorem \ref{gkdvmaintheorem}.

\begin{proof}[Proof of Theorem \ref{gkdvmaintheorem}] The proof follows classical arguments as the ones in \cite{bss} and \cite{Grillakis}. Assume by contradiction that  $\phi_k$ is  $H^{s_2/2}_{per}$-unstable. Then, we can choose $\varepsilon>0$ and initial data  $w_n:=u_n(0)\in U_{\frac{1}{n}}(\phi_k)$,   $n\in\N$,  such that
$$
\rho(w_n,\phi_k)\rightarrow 0 \quad\mbox{and}\quad \sup_{t\geq0}\rho(u_n(t),\phi_k)\geq \varepsilon,
$$
where $u_n(t)$ is the solution of  (\ref{gkdv}) with initial data  $w_n$. Here, by taking a smaller $\varepsilon$ if necessary,  we can assume that $\varepsilon>0$ is the one obtained in  Lemma \ref{lemacoercividadebbm}.  By the continuity of the solution in $t$, we can take the first time  $t_n>0$ such that
\begin{equation}
\rho(u_n(t_n),\phi_k)=\frac{\varepsilon}{2}.
\label{bbm16}
\end{equation}

The strategy now is to obtain a contradiction with  (\ref{bbm16}), for $n$ sufficiently large. Let $f_n$ be the function defined as
\[
\begin{split}
f_n(\alpha)&=M_k(\alpha u_n(t_n))\\
&=\alpha^2\frac{\partial c}{\partial k}\frac{1}{2}\int_0^L|u_n(t_n)|^2dx+\alpha\frac{\partial A}{\partial k}\int_0^Lu_n(t_n)dx. 
\end{split}
\]
Since $Q$ and $V$ are conserved quantities, we then see that
\begin{equation}\label{falene}
\begin{split}
f_n(\alpha)&=\alpha^2\frac{\partial c}{\partial k}\frac{1}{2}\int_0^L|w_n|^2dx+\alpha\frac{\partial A}{\partial k}\int_0^Lw_ndx\\
&=\alpha^2Q_k(w_n)+\alpha V_k(w_n),
\end{split}
\end{equation}
where we have denoted
$$
Q_k(u):=\frac{\partial c}{\partial k} Q(u)\quad \mbox{and}\quad V_k(u):=\frac{\partial A}{\partial k} V(u).
$$
On the other hand, since $\rho(w_n,\phi_k)\to0$, as $n\to\infty$, $Q$ and $V$ are invariant by translations and continuous, we have
\begin{equation}\label{bbm18}
Q_k(w_n)\longrightarrow Q_k(\phi_k)=:a,\qquad
V_k(w_n)\longrightarrow V_k(\phi_k)=:b.
\end{equation}
Hence, for any $\alpha\in\R$, we obtain from \eqref{falene},
\begin{equation}\label{convpont}
f_n(\alpha)\to f(\alpha),
\end{equation}
where
$$
f(\alpha)=\alpha^2Q_k(\phi_k)+\alpha V_k(\phi_k)=\alpha^2a+\alpha b.
$$

Before proceeding, we shall show that under assumption {\bf (H4)} there exist  real sequences $(\alpha_n)$ such that $M_k(\alpha_n u_n(t_n))=M_k(\phi_k)$.

\begin{lemma}\label{lemmaconv1}
Assume that {\bf (H4)} holds. We have the following.
\begin{itemize}
\item[(i)] If $\frac{\partial c}{\partial k}\neq0$ then there exist two sequences $(\alpha_n)$ and $(\tilde{\alpha}_n)$, and real numbers $\theta_0<\theta_1$ such that, for all $n$ sufficiently large,
\begin{equation}\label{alphaalphatilde}
M_k(\alpha_n u_n(t_n))=M_k(\tilde{\alpha}_n u_n(t_n))=M_k(\phi_k)
\end{equation}
and
\begin{equation}\label{alpha0alpha1}
\tilde{\alpha}_n\leq \theta_0<\theta_1\leq \alpha_n.
\end{equation}
In addition, up to a subsequence, either $(\alpha_n)$ or $(\tilde{\alpha}_n)$ converges to 1.
\item[(ii)] If $\frac{\partial c}{\partial k}=0$ then there exists a sequence $(\alpha_n)$  such that, for all $n\in \N$,
$$
M_k(\alpha_n u_n(t_n))=M_k(\phi_k).
$$
In addition, $(\alpha_n)$ converges to 1.
\end{itemize}
\end{lemma}
\begin{proof}
(i) Without loss of generality, we shall assume that  $\frac{\partial c}{\partial k}>0$. The case, $\frac{\partial c}{\partial k}<0$ can be treated exactly in the same manner with obvious modifications. First of all note that $a=Q_k(\phi_k)>0$. Thus the parabola $f$ has a minimum at $x_0=-b/2a$ with minimum value $f(x_0)=-b^2/4a$. In addition, note that
$$
f(0)=f\left( -\dfrac{b}{a}\right)=0, \qquad f(1)=f\left( -\dfrac{b+a}{a}\right)=a+b,
$$
and $a+b=M_k(\phi_k)$ is not the minimum value of $f$, otherwise we would have $a+b=-a$, contradicting assumption {\bf (H4)}.

Now, fix any real number $\beta$ satisfying $f(x_0)<\beta<a+b$ and let $K\subset\R$ be a compact set containing the interval $[x_0-1,x_0+1]$. Since $f_n\to f$ in $\R$, we see that $f_n\to f$ uniformly in $K$. Thus, there is $N\in\N$ such that for any $\alpha\in K$,
$$
n\geq N\quad \Rightarrow \quad |f_n(\alpha)-f(\alpha)|<\dfrac{a+b-\beta}{2}.
$$
Moreover, the continuity of $f$ at $x_0$ implies the existence of $\delta\in(0,1)$ such that
$$
|\alpha-x_0|<\delta\quad \Rightarrow \quad |f(\alpha)-f(x_0)|<\dfrac{a+b-\beta}{2}.
$$
Consequently, if $|\alpha-x_0|<\delta$ and $n\geq N$, we deduce
\[
\begin{split}
f_n(\alpha)&\leq |f_n(\alpha)-f(\alpha)|+|f(\alpha)-f(x_0)|+f(x_0)\\
&< a+b+f(x_0)-\beta\\
&<a+b.
\end{split}
\]
This mean that $a+b$ is not the minimum value of the parabola $f_n$. Hence, for $n\geq N$, there are real numbers $\tilde{\alpha}_n$ and $\alpha_n$ satisfying
$$
\tilde{\alpha}_n\leq x_0-\delta<x_0+\delta\leq \alpha_n,
$$
such that $f_n(\tilde{\alpha}_n)=f_n(\alpha_n)=a+b$. By taking $\theta_0=x_0-\delta$ and $\theta_1=x_0+\delta$ we obtain \eqref{alphaalphatilde} and \eqref{alpha0alpha1}.

It remains to show that either $(\alpha_n)$ or $(\tilde{\alpha}_n)$ admit a subsequence converging to 1. Since $Q$ and $V$ are conserved quantities, using \eqref{alphaalphatilde}, we have
\begin{equation}\begin{split}
\varrho& :=|\alpha_n^2Q_k(w_n)+\alpha_nV_k(w_n)-(Q_k(w_n)+V_k(w_n))|\\
&=|\alpha_n^2Q_k(u_n(t_n))+\alpha_nV_k(u_n(t_n))-(Q_k(w_n)+V_k(w_n))|\\
&=|Q_k(\alpha_nu_n(t_n))+V_k(\alpha_nu_n(t_n))-(Q_k(w_n)+V_k(w_n))|\\
&=|M_k(\alpha_nu_n(t_n))-M_k(w_n)|\\
&=|M_k(\phi_k)-M_k(w_n)|.
\end{split}
\label{bbm19}
\end{equation}

From \eqref{bbm18}, it follows that $\varrho\to0$, as $n\to\infty$.
Since
\begin{equation}\label{19.1}
\begin{split}
0&\leq |\alpha_n^2Q_k(w_n)+\alpha_nV_k(w_n)-(a+b)|\\
&\leq |\alpha_n^2Q_k(w_n)+\alpha_nV_k(w_n)-(Q_k(w_n)+V_k(w_n))|\\
&\quad + |(Q_k(w_n)+V_k(w_n))-(a+b)|\\
&\leq \varrho +|Q_k(w_n)-a| + |V_k(w_n)-b|,
\end{split}
\end{equation}
by using \eqref{bbm18},   we see that
\begin{equation}
z_n:=\alpha_n^2Q_k(w_n)+\alpha_nV_k(w_n) \longrightarrow a+b, \quad \mbox{as} \;\; n\rightarrow \infty.
\label{bbm20}
\end{equation}

Next we claim that $(\alpha_n)$ is a bounded sequence. On the contrary, suppose  $(\alpha_n)$ is unbounded. Because $Q_k(w_n)>0$ and $Q_k(w_n)$ and $V_k(w_n)$ are bounded  we obtain, up to a subsequence,
$$ 
z_n=\alpha_n(\alpha_nQ_k(w_n)+V_k(w_n))\longrightarrow +\infty, \quad \mbox{as} \;\; n\rightarrow \infty,$$
which contradicts \eqref{bbm20}. 

It is clear that by defining $\tilde{z}_n:=\tilde{\alpha}_n^2Q_k(w_n)+\tilde{\alpha}_nV_k(w_n)$, the same analysis can be performed to conclude that the sequence $(\tilde{\alpha}_n)$ is also bounded.

Therefore, there are subsequences of  $(\alpha_n)$ and $(\tilde{\alpha}_n)$, which we still denote by $(\alpha_n)$ and $(\tilde{\alpha}_n)$ such that
$$
\alpha_n\longrightarrow \alpha_0, \quad \mbox{and} \quad \tilde{\alpha}_n\longrightarrow \tilde{\alpha}_0, \quad \mbox{as} \;\; n\rightarrow \infty.
$$
Taking the limit in $z_n$ and $\tilde{z}_n$, it follows from \eqref{bbm18}  that
$\alpha_0^2a+\alpha_0b=a+b$ and $\tilde{\alpha_0}^2a+\tilde{\alpha}_0b=a+b$. 
This implies that
$$
\alpha_0=1 \quad \mbox{or}\quad \alpha_0=-\frac{b+a}{a}
$$
and 
$$
\tilde{\alpha}_0=1 \quad \mbox{or}\quad \tilde{\alpha}_0=-\frac{b+a}{a}.
$$
The inequalities \eqref{alpha0alpha1} imply that both sequences cannot converge to the same number. Thus we have either $\alpha_0=1$ or $\tilde{\alpha}_0=1$.

(ii) In this case, we have $a=0$ and, in view of assumption {\bf (H4)}, $b\neq0$. Thus $f_n$ and $f$ are linear functions passing through the origin. It is then clear that there is a sequence $(\alpha_n)$ such that $M_k(\alpha_n u_n(t_n))=f_n(\alpha_n)=b=M_k(\phi_k)$. As before, we conclude that $z_n:=\alpha_nV_k(w_n)\to b$, as $n\to\infty$. Thus,
\begin{equation}\label{alconc0}
|\alpha_n-1||V_k(w_n)|\leq |z_n-b|+|V_k(w_n)-b|.
\end{equation}
Since the right-hand side of \eqref{alconc0} goes to zero and $V_k(w_n)\to b\neq0$, we obtain that $\alpha_n \to1$. The proof of the lemma is thus completed.
\end{proof}

Next, we turn to the proof of Theorem \ref{gkdvmaintheorem} and assume, without loss of generality, that the sequence $(\alpha_n)$ converges to $1$.  First we prove the following two claims.

\vskip.3cm
\noindent {\bf{Claim 1:}} $\rho(u_n(t_n),\alpha_nu_n(t_n))\longrightarrow 0$, as $ n\rightarrow \infty.$

In fact, by definition,
\begin{equation}
\begin{split}
\ds\rho(u_n(t_n),\alpha_nu_n(t_n))&\ds=\inf_{r\in\mathbb{R}}\|u_n(\cdot,t_n)-\alpha_nu_n(\cdot+r,t_n)\|_{H^{s_2/2}_{per}}\\
&\leq \|u_n(t_n)-\alpha_nu_n(t_n)\|_{H^{s_2/2}_{per}}=|(1-\alpha_n)|\|u_n(t_n)\|_{H^{s_2/2}_{per}}.\\
\end{split}
\label{bbm21}
\end{equation}
On the other hand, since $\ds\rho(u_n(t_n),\phi_k)=\frac{\varepsilon}{2}$, there exists $r\in\mathbb{R}$ such that
$$
\|u_n(t_n)\|_{H^{s_2/2}_{per}}\leq \|u_n(t_n)-\phi_k(\cdot+r)\|_{H^{s_2/2}_{per}}+\|\phi_k(\cdot+r)\|_{H^{s_2/2}_{per}}<\varepsilon+\|\phi_k(\cdot+r)\|_{H^{s_2/2}_{per}},
$$
which is to say that the sequence $(\|u_n(t_n)\|_{H^{s_2/2}_{per}})$ is uniformly bounded. Taking the limit in \eqref{bbm21}, as $n\to\infty$, and taking into account that $\alpha_n\longrightarrow 1$, we obtain the claim.

\vskip.3cm
\noindent {\bf{Claim 2:}} $\rho(\alpha_nu_n(t_n),\phi_k)\longrightarrow 0$, as $n\rightarrow\infty$.

In fact, from Claim 1 and (\ref{bbm16}), we see that
$$
\rho(\alpha_nu_n(t_n),\phi_k)\leq \rho(\alpha_nu_n(t_n),u_n(t_n))+\rho(u_n(t_n),\phi_k)<\frac{\varepsilon}{3}+\frac{\varepsilon}{2}=\frac{5\varepsilon}{6}<\varepsilon,
$$
for all $n$ large enough. This means that for $n$ large enough, $\alpha_nu_n(t_n)\in U_\varepsilon(\phi_k)$. Moreover, Lemma \ref{lemmaconv1} implies that $\alpha_nu_n(t_n)\in\Sigma_k$. By putting all this together, we obtain that   $\alpha_nu_n(t_n)\in\Sigma_k\cap U_{\varepsilon}(\phi_k)$. Consequently,   Lemma \ref{lemacoercividadebbm} implies
\[
\begin{split}
\rho(\alpha_nu_n(t_n),\phi_k)^2&\leq C|F_k(\alpha_nu_n(t_n))-F_k(\phi_k)|\\
&\leq C|F_k(\alpha_nu_n(t_n))-F_k(u_n(t_n))|+ C|F_k(u_n(t_n))-F_k(\phi_k)|\\
&=C|F_k(\alpha_nu_n(t_n))-F_k(u_n(t_n))|+ C|F_k(w_n)-F_k(\phi_k)|.
\end{split}
\]
Taking the limit, as $n\to\infty$, in the last inequality, the continuity of $F_k$ and the boundedness of $(\alpha_nu_n(t_n))_{n\in\mathbb{N}}$ in $H^{s_2/2}_{per}([0,L])$
 yield Claim 2.
\vskip.3cm

Finally, Claims 1 and 2 combine to give
$$
\frac{\varepsilon}{2}=\rho(u_n(t_n),\phi_k)\leq \rho(u_n(t_n),\alpha_nu_n(t_n))+\rho(\alpha_nu_n(t_n),\phi_k) \longrightarrow 0,
$$
as $n\to\infty$, which is a contradiction. The proof of Theorem \ref{gkdvmaintheorem} is thus established.
\end{proof}

\section{Applications}\label{applic}

In this section we apply Theorem \ref{gkdvmaintheorem} to prove the orbital stability, in the energy space, for some well known dispersive models. As is well known, one of the major difficulties in the theory is to check that the spectral properties assumed in  {\bf (H1)-(H2)} hold. In many of the interesting applications the function $f$ in \eqref{gkdv} is a power or a polynomial. Here and in what follows we shall assume that it has this form.

Let us brief recall the main spectral properties of $\Lk$. The spectrum of $\mathcal{L}_k$ is formed by a sequence of eigenvalues, say, $\{\lambda_m\}_{m=0}^\infty$ satisfying $\lambda_0\leq\lambda_1\leq\lambda_2\leq\ldots$, and $\lambda_m\to\infty$, as $m\to\infty$, where equality means multiplicity of an eigenvalue (see e.g., \cite{AN2} and \cite{ea}). From \eqref{soleq} it easily follows that zero is an eigenvalue with associated eigenfunction $\phi_k'$. Thus, the task in general is to show that $\lambda_0<\lambda_1=0<\lambda_2$. Below we recall three different ways of determining this relation.

(i) {\it Lam\'e's type potential}. 
In many situations when $\mathcal{M}$ is a second order differential operator and the periodic traveling wave under study depends on the Jacobian elliptic functions, $\mathcal{L}_k$ turns out to be a Hill's operator with a Lam\'e type potential. In particular, studying the spectrum of $\mathcal{L}_k$ is equivalent to studying the eigenvalue problem 
\begin{equation}\label{specproblem}
\left\{\begin{array}{l}
\ds\Lambda''(x)+\left[h-n(n+1)\cdot k^2 \sn^2\left(x,k\right)\right]\Lambda(x)=0,\\
\Lambda(0)=\Lambda(2K(k)), \quad \Lambda'(0)=\Lambda'( 2K(k)),
\end{array}\right.
\end{equation}
where $h$ is a real parameter and $n$ is a non-negative integer. Depending on $n$, the first eigenvalues of \eqref{specproblem} are well known (see e.g., \cite{ince}). Many applications using this approach have appeared in the literature (see e.g., \cite{angulo5}, \cite{angulo3}, \cite{angulo1}, \cite{boka}, \cite{depas}, \cite{hik}, \cite{kade}, \cite{NP1}, \cite{natali-pastor} to cite but a few).

(ii) {\it Neves' approach.} Assume that $\mathcal{M}$ is a second order differential operator. Let us first recall from Floquet's theorem (see e.g., \cite{Magnus} page 4) that if $y$ is any solution of $\Lk y=0$, linearly independent of $\phi_k'$, then there exists a constant $\theta$ satisfying 
\begin{equation}\label{thetadef}
y(x+L)=y(x)+\theta \phi_k'(x).
\end{equation}
In particular, if $y$ satisfies the initial condition  $y'(0)=0$ then
by taking the derivative with respect to $x$ in both sides of \eqref{thetadef} and evaluating the result at $x=0$, we see that
\begin{equation}\label{thetadef1}
\theta=\dfrac{y'(L)}{\phi_k''(0)}.
\end{equation}
Under these conditions Theorem 3.1 in \cite{Neves1} (see also \cite{natali2}) states that $\lambda_1$ is simple if and only if $\theta\neq0$. In addition, $\lambda_1=0$ if and only if $\theta<0$. 

(iii) {\it Angulo and Natali's approach}.
When $f(x)=x^p$, for some integer $p\geq1$, a different approach to check {\bf (H1)-(H2)} was established in \cite{AN2}. Such an approach is based on the total positivity theory (see e.g., \cite{kar}) and can be viewed as an extension to the periodic case of the results in \cite{alb} and \cite{albbo}. To give the precise statement, we recall that a sequence $\{\alpha_n\}_{n\in\Z}$ of real numbers is said to be in the class $PF(2)$ discrete if
\begin{itemize}
\item[(i)] $\alpha_n>0$, for all $n\in\Z$;
\item[(ii)] $\alpha_{n_1-m_1}\alpha_{n_2-m_2}-\alpha_{n_1-m_2}\alpha_{n_2-m_1}>0$, for $n_1<n_2$ and $m_1<m_2$.
\end{itemize}

Assume that our assumptions on the operator $\mathcal{M}$ hold. Suppose in addition that $\phi_k$ is positive, even and such that $\widehat{\phi}_k>0$ and $\widehat{\phi_k^p}$ belongs to the class $PF(2)$ discrete, then $\mathcal{L}_k$ satisfies {\bf (H1)-(H2)} (see \cite[Theorem 4.1]{AN2}).

Next we will give some applications of Theorem \eqref{gkdvmaintheorem}.

\subsection{The KdV equation}

This subsection is devoted to the study of the KdV equation
\begin{equation}\label{kdvkdv}
u_t+u_{xxx}+\partial_x\left(\frac{u^2}{2}\right)=0,
\end{equation}
which appears as an approximated equation for the propagation of unidirectional, one-dimensional, small-amplitude long waves in a nonlinear dispersive media and it was firstly derived by Korteweg and de Vries in \cite{kdv}.

Since $\mathcal{M}=-\partial_x^2$. It is clear that our assumption on $\mathcal{M}$ are fulfilled with $s_1=s_2=2$ and $\gamma=0$. Hence, our energy space is the Sobolev space of order 1. As an application of the quadrature method, it is well known that \eqref{kdvkdv} has a periodic traveling-wave solutions $u(x,t)=\phi(x-ct)$ with
\begin{equation}\label{cnkdv}
\phi(y)=\phi_k(y)=12k^2b^2\cn^2(by,k), \qquad k\in(0,1).
\end{equation}
This means that $\phi_k$ is a solution of
\begin{equation}\label{appkdv}
-\phi_k+c\phi_k-\frac{1}{2}\phi_k^2+A=0,
\end{equation}
where
$$
c=4b^2(2k^2-1) \quad \mbox{and} \quad  A=24b^4k^2(1-k^2).
$$
Here $b\in\R$ is an arbitrary parameter. In order to obtain periodic solutions with a fixed period $L>0$, for any $k\in(0,1)$, we shall take
$$
b:=\frac{2K(k)}{L}.
$$
Since $\cn^2$ has fundamental period $2K(k)$, with this choice of $b$, the functions in \eqref{cnkdv} turn out to be $L$-periodic.
Also, to have our assumption $c>-\gamma=0$ in {\bf(H0)}, we need to restrict the elliptic modulus to the interval $J:=(k^*,1)$, where $k^*=\sqrt{2}/2$. This constructions then give the family of $L$-periodic solutions
$$
k\in J=(k^*,1)\mapsto \phi_k\in H^2_{per}([0,L]).
$$
The assumption  {\bf (H0)} is thus fulfilled.

Let us check {\bf (H1)} and {\bf (H2)}. Here we have $\Lk=-\partial_x^2+c-\phi_k$. It is not difficult to see that studying the spectral problem
\begin{equation}\label{specproblemkdv}
\left\{\begin{array}{l}
\ds\Lk f=\lambda f,\\
f(0)=f(L), \quad f'(0)=f'(L),
\end{array}\right.
\end{equation}
is equivalent to study the problem
\begin{equation}\label{specproblemkdv1}
\left\{\begin{array}{l}
\ds\Lambda''(x)+\left[h-12\cdot k^2 \sn^2\left(x,k\right)\right]\Lambda(x)=0,\\
\Lambda(0)=\Lambda(2K(k)), \quad \Lambda'(0)=\Lambda'( 2K(k)),
\end{array}\right.
\end{equation}
where $h=(12k^2b^2+\lambda-c)/b^2$.
It is well known that the first three eigenvalues of \eqref{specproblemkdv1} are simple and given by (see \cite{angulo1} and \cite{ince})
$$
h_0 = 2+5k^2 -2 \sqrt{1 - k^2 + 4k^4},\quad h_1=4+4k^2, \quad h_2 = 2+5k^2 +2 \sqrt{1 - k^2 + 4k^4}.
$$
Note that $h_1$ is an eigenvalue of \eqref{specproblemkdv1} if and only if $\lambda_1=0$ is an eigenvalue of \eqref{specproblemkdv}. Since $h_0<h_1<h_2$, the relation between $h$ and $\lambda$ then implies that $\lambda_1=0$ is a simple eigenvalue of $\Lk$, which in turn also implies that $\Lk$ has a unique negative eigenvalue. Assumptions {\bf (H1)-(H2)} are thus checked.

To check {\bf (H3)} we differentiate \eqref{appkdv} to see that 
$$
\Phi:=\left\langle\mathcal{ L}_k\left(\frac{\partial \phi_k}{\partial k}\right),\frac{\partial \phi_k}{\partial k}\right\rangle=-\frac{1}{2}\frac{\partial c}{\partial k}\frac{d}{dk}\left(\int_0^L\phi_k^2\,dx\right)-\frac{\partial A}{\partial k}\frac{d}{dk}\left(\int_0^L\phi_k\,dx\right).
$$
Using formulae 312.02 and 312.04 in \cite{friedman}, we obtain
$$
\int_0^L\phi_k(x)dx=\frac{48K(k)}{L}\Big(E(k)-(1-k^2)K(k)\Big)
$$
and
$$
\int_0^L\phi_k^2(x)dx=\frac{48^2K^3(k)}{3L^3}\Big( (2-5k^2+3k^2)K(k)+(4k^2-2)E(k) \Big).
$$
Thus, using the expressions for $c$ and $A$ we see that $\Phi<0$ is equivalent to
\[
\begin{split}
\frac{16\cdot 48^2}{3L^3}\frac{d}{dk}& \Big((2k^2-1)K(k)\Big)\frac{d}{dk}\Big( (2-5k^2+3k^2)K(k)^4+(4k^2-2)E(k)K(k)^3 \Big)\\
&+\frac{18432}{L^5}\frac{d}{dk}\Big( k^2(1-k^2)K(k)^4\Big)\frac{d}{dk}\Big(E(k)K(k)-(1-k^2)K(k)^2\Big)>0.
\end{split}
\]
The positivity of this quantity can be checked numerically (easy) or analytically (hard) using Taylor expansions of the elliptic functions (see \cite{depas} for similar calculations). This shows {\bf (H3)}.

Finally, one can easily verify {\bf (H4)} by noting that $c$ and $A$ have positive derivatives and $\phi_k$ is non-negative. 

Thus an application of Theorem \ref{gkdvmaintheorem} gives the following result.

\begin{theorem}\label{kdvstabteo}
For each $k\in J=(k^*,1)$, the periodic traveling wave $\phi_k$ given in \eqref{cnkdv}
 is orbitally stable  in $H^{1}_{per}([0,L])$.
\end{theorem}

The result in Theorem \ref{kdvstabteo} is not new; the orbital stability of the cnoidal waves has already appeared in \cite{angulo1} and \cite{deka} (see also \cite{boka}, \cite{johnson1},  \cite{Neves1}). In \cite{angulo1} the authors first show that it is possible to choose the constant $A$ in \eqref{appkdv} such that the corresponding family of cnoidal waves (which were written in a different way from that in \eqref{cnkdv})  has mean zero over its fundamental period. Then, by a translation they show that it is possible to obtain a family with a fixed mean. The orbital stability with respect to small perturbations in $H^1_{per}([0,L])$ was then obtained as an adaptation of the ideas in \cite{Grillakis}.  In \cite{deka}, the authors take the advantage of the integrability of the KdV equation (and the results in \cite{boka}) to show the orbital stability of the cnoidal waves with respect to subharmonic perturbations which respect the mean value of the solution. Here, for a fixed $n\in\N$, subharmonic perturbations means perturbations in the space $H^2_{per}([0,nL])$.

\subsection{The modified KdV equation}\label{subsec3.2}

This subsection is devoted to study the modified Korteweg-de Vries (mKdV) equation
\begin{equation}
u_t+u_{xxx}+\partial_x(2u^3)=0.
\label{mkdv}
\end{equation}
The constant $2$ in $f(u)=2u^3$ appears only for convenience.
The periodic traveling waves are also those solutions of the form $u(x,t)=\phi(x-ct)$, where $\phi$ must be a solution of the ODE
\begin{equation}
\phi ''-c\phi+2\phi^3-A=0.
\label{mkdvedo}
\end{equation}
As in \eqref{kdvsoleq1}, multiplying \eqref{mkdvedo} by $\phi'$ and integrating once we obtain
\begin{equation}
(\phi')^2=-\phi^4+c\phi^2+2A\phi+B=:P(\phi),
\label{mkdv1}
\end{equation}
where $B$ another integration constant. It is then seen that the solutions of \eqref{mkdvedo} depend on the roots of the polynomial $P$.
We will study two particular family of solutions of \eqref{mkdvedo}.

\subsubsection{Dnoidal type solutions}

By assuming  that $A=0$ and $P$ has four real roots (note that $P$ is an even polynomial in this case), \eqref{mkdvedo} admits a family of solutions given by
\begin{equation}\label{dnmkdv}
\phi(y)=\phi_k(y)=a\dn(ax,k), \qquad k\in(0,1),
\end{equation}
where $a>0$ is an arbitrary constant and 
\begin{equation}\label{cmkdv}
c=a^2(2-k^2).
\end{equation}
In order to obtain periodic solutions with a fixed period $L>0$, for any $k\in(0,1)$, we now take
$$
a:=\frac{2K(k)}{L}.
$$
By recalling that $\dn$ has fundamental period $2K(k)$ we then see that the functions in \eqref{dnmkdv} have fundamental period $L$. Since $c>0$ for any $k\in(0,1)$, the assumption {\bf (H0)} holds with $J=(0,1)$.

Let us check {\bf (H1)} and {\bf (H2)}. The linearized operator reads as  $\Lk=-\partial_x^2+c-6\phi_k^2$ and the spectral problem
\begin{equation}\label{specproblemmkdv}
\left\{\begin{array}{l}
\ds\Lk f=\lambda f,\\
f(0)=f(L), \quad f'(0)=f'(L),
\end{array}\right.
\end{equation}
is equivalent to 
\begin{equation}\label{specproblemmkdv1}
\left\{\begin{array}{l}
\ds\Lambda''(x)+\left[h-6\cdot k^2 \sn^2\left(x,k\right)\right]\Lambda(x)=0,\\
\Lambda(0)=\Lambda(2K(k)), \quad \Lambda'(0)=\Lambda'( 2K(k)),
\end{array}\right.
\end{equation}
where $h=(6a^2+\lambda-c)/a^2$.
It is well known that the first three eigenvalues of \eqref{specproblemmkdv1} are simple and given by (see \cite{angulo3} and \cite{ince})
$$
h_0 = 2(1+k^2- \sqrt{1 - k^2 + k^4}),\quad h_1=4+k^2, \quad h_2 = 2(1+k^2+ \sqrt{1 - k^2 + k^4}).
$$
It is easy to see that $h_1$ correspond to the eigenvalue $\lambda_1=0$  of \eqref{specproblemmkdv}. Since $h_0<h_1<h_2$, the relation between $h$ and $\lambda$ then implies that $\lambda_1=0$ is a simple eigenvalue and $\Lk$ has a unique negative eigenvalue. Assumptions {\bf (H1)-(H2)} are hence checked.

Next, since $A=0$ we see from \eqref{mkdv1} that 
$$
\Phi:=\left\langle\mathcal{ L}_k\left(\frac{\partial \phi_k}{\partial k}\right),\frac{\partial \phi_k}{\partial k}\right\rangle=-\frac{1}{2}\frac{\partial c}{\partial k}\frac{d}{dk}\left(\int_0^L\phi_k^2\,dx\right)
$$
Using formulas 314.02 and 312.04 in \cite{friedman}, we obtain
$$
\int_0^L\phi_k^2(x)dx=\frac{8K(k)E(k)}{L}.
$$
Since $k\mapsto K(k)E(k)$ is a strictly increasing function and $\frac{\p c}{\p k}>0$ we obtain $\Phi<0$ and assumption {\bf (H3)} is fulfilled. It is clear that {\bf (H4)} holds. As a consequence of Theorem \ref{gkdvmaintheorem}, we obtain the following theorem.

\begin{theorem}\label{mkdvstabteo}
For each $k\in J=(0,1)$, the periodic traveling wave $\phi_k$ given in \eqref{dnmkdv}
 is orbitally stable  in $H^{1}_{per}([0,L])$.
\end{theorem}

To the best of our knowledge, Theorem \ref{mkdvstabteo} has first appeared in \cite{angulo3} where the author exploited his results established for the cubic Schr\"odinger equation. The approach to obtain the results was based on the classical ideas contained \cite{be}, \cite{bona}, and \cite{weinstein1}. From the point of view of Hamiltonian systems, the same result was established in \cite{kade}.

\subsubsection{Dnoidal-Snoidal type solutions}\label{dssub}
We now assume that zero is a root of the polynomial $P$ in \eqref{mkdv1}. This immediately implies that the integration constant $B$ must be zero. By assuming that $\alpha_1<0<\alpha_3<\alpha_4$ are the roots of $P$, applying the quadrature method  and using  formula 257.00 in \cite{friedman}, we obtain the following $L$-periodic smooth curve of solutions for \eqref{mkdvedo},
$$
\phi_{k}(\xi)=\frac{\alpha_4(k)\left(\alpha_3(k)-\alpha_1(k)\right)\dn^2\left(\frac{2K(k)}{L}\xi,k\right)}{\left(\alpha_3(k)-\alpha_1(k)\right)+\left(\alpha_4(k)-\alpha_3(k)\right)\sn^2\left(\frac{2K(k)}{L}\xi,k\right)}.
$$
where the constants $\alpha_i(k)$ are given by (after some algebra)
$$\begin{array}{l}
\ds\alpha_1(k)=\frac{-2K(k)}{L}\left(\sqrt{2\sqrt{k^4-k^2+1}+1-2k^2}+\frac{1}{\sqrt{3}}\sqrt{2\sqrt{k^4-k^2+1}-1+2k^2}\right),\\
\\
\ds \alpha_3(k)= \frac{2K(k)}{L}\left(\sqrt{2\sqrt{k^4-k^2+1}+1-2k^2}-\frac{1}{\sqrt{3}}\sqrt{2\sqrt{k^4-k^2+1}-1+2k^2}\right),\\
\\
\ds \alpha_4(k)=\frac{4K(k)}{\sqrt{3}L}\sqrt{2\sqrt{k^4-k^2+1}-1+2k^2}.
\end{array}
$$
In addition $c$ and $A$ can also be expressed in terms of $k$ as 
$$
\ds c(k)=\frac{16K^2(k)}{L^2}\sqrt{k^4-k^2+1}
$$
and
$$
 A(k)=\frac{-32K^3(k)}{3\sqrt{3}L^3}\left(\sqrt{k^4-k^2+1}-2k^2+1\right)\sqrt{2\sqrt{k^4-k^2+1}+2k^2-1}.
 $$

After a few algebraic manipulations involving elliptic functions $\phi_k$ can be written as
\begin{equation}
\phi_k(\xi)=\frac{4K(k)}{\sqrt{2}g(k)L}\left(\frac{\dn^2\left(\frac{2K(k)}{L}\xi,k\right)}{1+\beta^2\sn^2\left(\frac{2K(k)}{L}\xi,k\right)}\right),
\label{mkdv29}
\end{equation}
where $\beta^2=\sqrt{k^4-k^2+1}+k^2-1$ and $g(k)=\sqrt{\sqrt{k^4-k^2+1}-k^2+\frac{1}{2}}$. Consequently, we obtain {\bf (H0)} with $J=(0,1)$.

\begin{remark} Formally, by making $k\to1$ in \eqref{mkdv29}, the periodic function $\phi_k$ looses its periodicity and degenerates to
$$
\phi_0(\xi)=\frac{-\alpha_1\alpha_4}{-\alpha_4+\frac{(\alpha_4-\alpha_1)}{2}+\frac{(\alpha_4-\alpha_1)}{2}\mbox{cosh}\left(2 \xi\right)},
$$
which is a solitary wave for the mKdV equation.
Orbital stability in the energy space of this solution was studied in \cite{alejo1}.
\end{remark}

In order to check {\bf (H1)-(H2)} we will take the advantage of the results in \cite{Neves1}, \cite{Neves2}, and \cite{natali2} as briefly described at the beginning of this section. Let $(n,z)$ be the inertial index of the linearized operator $\Lk=-\partial_x^2+c(k)-6\phi_k^2$, that is, $n$ denotes the dimension
of the negative subspace of $\Lk$ and $z$ denotes the dimension of $\ker(\Lk)$. Under our constructions, $\Lk$ is \textit{isonertial}, that is, $(n,z)$ does not depend on $k$ and on $L$ (see Theorem 3.1 in \cite{Neves2} or Theorem 3.1 \cite{natali2}). Thus, it suffices to fix $k\in(0,1)$ and $L>0$. For the sake of simplicity, we fix $k_0:=0.5$ and $L_0=30$. Since $\phi_{k_0}'$ has two zeros in the interval $[0,L_0)$ and 
$$
\mathcal{L}_k(\phi_k')= - \phi_k'''+c\phi'_k -6\phi_k^2\phi_k',=\left(- \phi_k''+c(k)\phi_k-2\phi_k^3 \right)'= 0.
$$
we obtain that zero is the second or the third eigenvalue of $\mathcal{L}_{k_0}$. Theorem 3.2 in \cite{natali2} establishes that zero is the second eigenvalue (which in turn is simple) provided that
\begin{equation}
\theta:=\frac{y'(L_0)}{\phi_{k_0}''(0)}<0,
\label{formulathetamkdv}
\end{equation}
where $y$ is the unique solution of the IVP
\begin{equation}
\left\{\begin{array}{l}
\ds-y''+\left[c(k_0)-6\phi_{k_0}^{2}\right]y=0,\\
y(0)=-\frac{1}{\phi_{k_0}''(0)},\\
y'(0)=0.
\end{array}\right.
\label{pvialoisiomkdv}
\end{equation}
The sign of $\theta$ can be obtained once we have the value $y'(L_0)$. We can solve (numerically) \eqref{pvialoisiomkdv} and, in particular, we deduce that 
$$\theta\cong-1.382078401\times 10^5.$$
Next table illustrates some values of $\theta$ if we fix $k_0$ and choose different values of $L$.

\vskip.3cm
\begin{center}
\begin{tabular}{|c|c|c|c|c|}
\hline 
\multicolumn{5}{|c|}{Values of $\theta$ with  $k_0=0.5$. }\\ 
\hline 
$L=20$ & $L=50$ & $L=200$ & $L=1000$ & $L=1000000$ \\ 
\hline 
$\theta\cong -18200$ & $\theta\cong -1.77 \times 10^6$ & $\theta\cong -1.82 \times 10^9$ & $\theta\cong -5.68 \times 10^{12}$ & $\theta\cong -5.68 \times 10^{27}$ \\ 
\hline 
\end{tabular} \\
\end{center}
\vskip.3cm
This checks {\bf (H1)-(H2)}.

We now check {\bf (H3)}. By deriving  equation (\ref{mkdvedo}), with respect to $k$, we get
$$\mathcal{L}_k \left(\frac{\partial \phi_k}{\partial k}\right)=-\frac{\partial c}{\partial k}\phi_k-\frac{\partial A}{\partial k}.$$
Thus, we can write $\Phi$ as
$$\begin{array}{rl}
\Phi=\ds&-\ds\left\langle \frac{\partial c}{\partial k}\phi_k+\frac{\partial A}{\partial k} , \frac{\partial \phi_k}{\partial k}\right\rangle=-\left\langle M'_k(\phi_k),\frac{\partial \phi_{k}}{\partial k}\right\rangle\\
\\
=&-\ds\int_0^L \frac{\partial c}{\partial k}\phi_k\frac{\partial \phi_k}{\partial k}+ \frac{\partial A}{\partial k}\frac{\partial \phi_k}{\partial k} dx\\
\\
=&-\ds\ds\int_0^L\frac{1}{2} \frac{\partial c}{\partial k} \frac{\partial }{\partial k}\phi_k^2 + \frac{\partial A}{\partial k}\frac{\partial \phi_k}{\partial k} dx.\\
\end{array}
$$
In other words,
$$\begin{array}{rl}
\Phi=\ds&-\ds  \frac{\partial c}{\partial k} \frac{\partial }{\partial k}\left(\frac{1}{2}\int_0^L \phi_k^2dx\right) - \frac{\partial A}{\partial k}\frac{\partial}{\partial k}\int_0^L\phi_k dx\\
\\
=&\ds- \frac{\partial c}{\partial k} \frac{\partial }{\partial k}Q(\phi_k) - \frac{\partial A}{\partial k}\frac{\partial}{\partial k}V(\phi_k).
\end{array}
$$

Since $\phi_k$ is given in \eqref{mkdv29}, we can compute  $Q(\phi_k)$ and $V(\phi_k)$ to obtain $\Phi$. Indeed,
$$
\begin{array}{ll}
V(\phi_k)=\ds \int_0^L\phi_k(x)dx & =\ds \int_0^L \frac{4K(k)}{\sqrt{2}g(k)L}\left(\frac{\dn^2\left(\frac{2K(k)}{L}x,k\right)}{1+\beta^2\sn^2\left(\frac{2K(k)}{L}x,k\right)}\right) dx\\
\\
&=\ds \frac{4K(k)}{\sqrt{2}g(k)L} \int_0^L  \frac{\dn^2\left(\frac{2K(k)}{L}x,k\right)}{1+\beta^2\sn^2\left(\frac{2K(k)}{L}x,k\right)} dx.\\
\end{array}
$$
By making the change of variable $\ds y=\frac{2K(k)}{L}x$, we obtain
$$V(\phi_k)=\frac{4K(k)}{\sqrt{2}g(k)L}\frac{L}{2K(k)}\int_0^{2K(k)}\frac{\dn^2(y,k)}{1+\beta^2\sn^2(y,k)}dy,
$$
that is,
$$V(\phi_k)=\frac{4}{\sqrt{2}g(k)}\int_0^{K(k)}\frac{\dn^2(y,k)}{1+\beta^2\sn^2(y,k)}dy=:\frac{4}{\sqrt{2}g(k)} I_1.
$$
With similar arguments, we see that
$$Q(\phi_k)=\frac{4K(k)}{g^2(k)L}\int_0^{K(k)}\frac{\dn^4(y,k)}{(1+\beta^2\sn^2(y,k))^2}dy=:\frac{4K(k)}{g^2(k)L} I_2.
$$

Now, let $\alpha^2$ be such that $-\alpha^2=\beta^2$ it follows that $0<-\alpha^2<k^2$ and  formula 410.04 in \cite{friedman} implies that
$$
I_1=\frac{(k^2-\alpha^2)G(w,k)}{\sqrt{\alpha^2(1-\alpha^2)(\alpha^2-k^2)}},
$$
where, for $k'=\sqrt{1-k^2}$,
$$G(w,k)=K(k)E(w,k')-K(k)F(w,k')+E(k)F(w,k')=\frac{\pi}{2}\Lambda_0(w,k)$$
and
$$w=\sin^{-1}\sqrt{\frac{\alpha^2}{\alpha^2-k^2}}=\sin^{-1}\frac{\beta}{\sqrt{k^2+\beta^2}}.$$ The functions $F$ and $E$ are the incomplete elliptic integral of the first and second kind and the function $\Lambda_0(w,k)$ is known as  \textit{Heuman's Lambda function} (see e.g., \cite{friedman} for additional details).  By using the expression $\beta^2=\sqrt{k^4-k^2+1}+k^2-1$, we can rewrite
$$
I_1=\frac{\sqrt{\sqrt{k^4-k^2+1}+2k^2-1}G(w,k)}{\sqrt{2k^4-2k^2+1+(2k^2-1)\sqrt{k^4-k^2+1}}},
$$
and
$$
w=\sin^{-1}\sqrt{\frac{\sqrt{k^4-k^2+1}+k^2-1}{\sqrt{k^4-k^2+1}+2k^2-1}}.
$$

On the other hand, the relation $\dn^2=1-k^2\sn^2$ yields
$$I_2=\int_0^{K(k)}\frac{(1-k^2\sn^2(y))^2}{(1+\beta^2\sn^2(y))^2}dy.$$
Thus, by  formula  410.8 in \cite{friedman}, we have
$$I_2=\frac{1}{\alpha^4}\Big(k^4K(k)+2k^2(\alpha^2-k^2)\Pi(\alpha^2,k)+(\alpha^2-k^2)^2V_2\Big),$$
where
$$\Pi(\alpha^2,k)=\frac{k^2K(k)}{k^2-\alpha^2}-\frac{\alpha^2G(w,k)}{\sqrt{\alpha^2(1-\alpha^2)(\alpha^2-k^2)}}$$
and
$$V_2=\eta_0\left(\alpha^2E(k)+ \frac{2k^4\alpha^2-2k^4+\alpha^4(k')^2}{k^2-\alpha^2}K(k)-\frac{\alpha^2(2\alpha^2k^2+2\alpha^2-\alpha^4-3k^2)G(w,k)}{\sqrt{\alpha^2(1-\alpha^2)(\alpha^2-k^2)}}\right),$$
with $\ds\eta_0=1/(2(\alpha^2-1)(k^2-\alpha^2))$ and $w$ as before.

In view of the above relations,
$$\begin{array}{rl}
\Phi=\ds&-\ds \frac{\partial c}{\partial k} \frac{\partial }{\partial k}\left( \frac{4K(k)}{g^2(k)L} I_2  \right) - \frac{\partial A}{\partial k}\frac{\partial}{\partial k}\left( \frac{4}{\sqrt{2}g(k)} I_1  \right) .
\end{array}
$$
Besides, 
$$\frac{\partial c}{\partial k}=\frac{1}{L^2}\frac{\partial }{\partial k}\left(16K^2(k)\sqrt{k^4-k^2+1}\right)=:\frac{1}{L^2}m_1(k)$$
and
$$\frac{\partial A}{\partial k}=\frac{1}{L^3}\frac{\partial }{\partial k}\left(\frac{-32K^3(k)\left(\sqrt{k^4-k^2+1}-2k^2+1\right)}{3\sqrt{3}}\sqrt{2\sqrt{k^4-k^2+1}+2k^2-1}\right)=: \frac{1}{L^3}m_2(k).$$
Thus, 
$$\begin{array}{rl}
\Phi=\ds&-\ds \frac{1}{L^2}m_1(k)\frac{\partial }{\partial k}\left( \frac{4K(k)}{g^2(k)L} I_2  \right) - \frac{1}{L^3}m_2(k)\frac{\partial}{\partial k}\left( \frac{4}{\sqrt{2}g(k)} I_1  \right)
\end{array}$$
and finally,
$$\Phi=-\frac{1}{L^3}m_3(k),$$
where
$$m_3(k)=m_1(k)\frac{\partial }{\partial k}\left( \frac{4K(k)}{g^2(k)} I_2  \right) + m_2(k)\frac{\partial}{\partial k}\left( \frac{4}{\sqrt{2}g(k)} I_1\right).$$
Since $m_3(k)$ depends only on $k$, we can check, at least numerically, that $m_3(k)>0$, for all $k\in(0,1)$ (see Figure \ref{mkmkdv1111} below). Therefore, we  conclude that $\Phi<0$.

\begin{figure}[htb]
	\centering
	\subfigure{\includegraphics[scale=0.25]{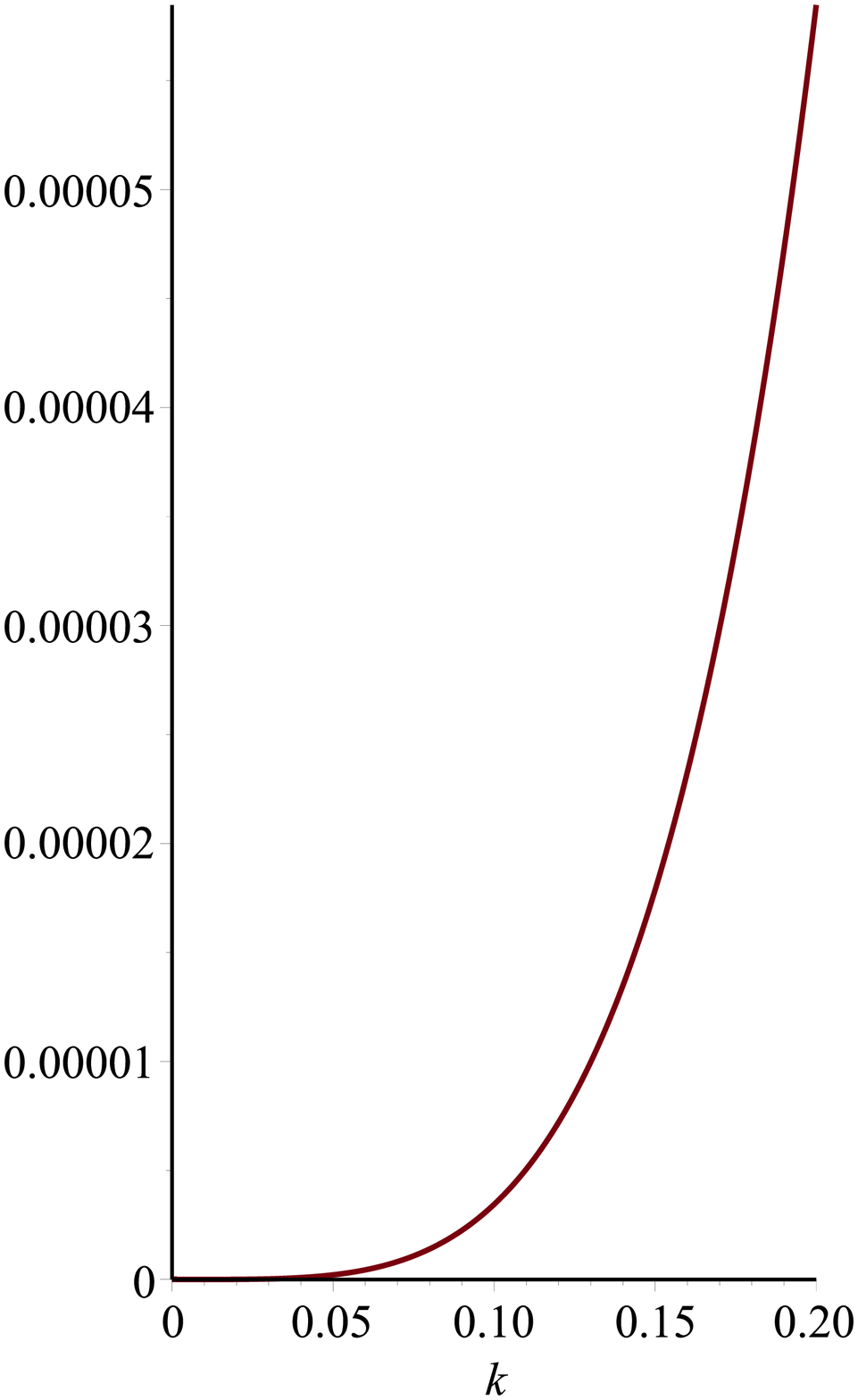}}
	\subfigure{\includegraphics[scale=0.25]{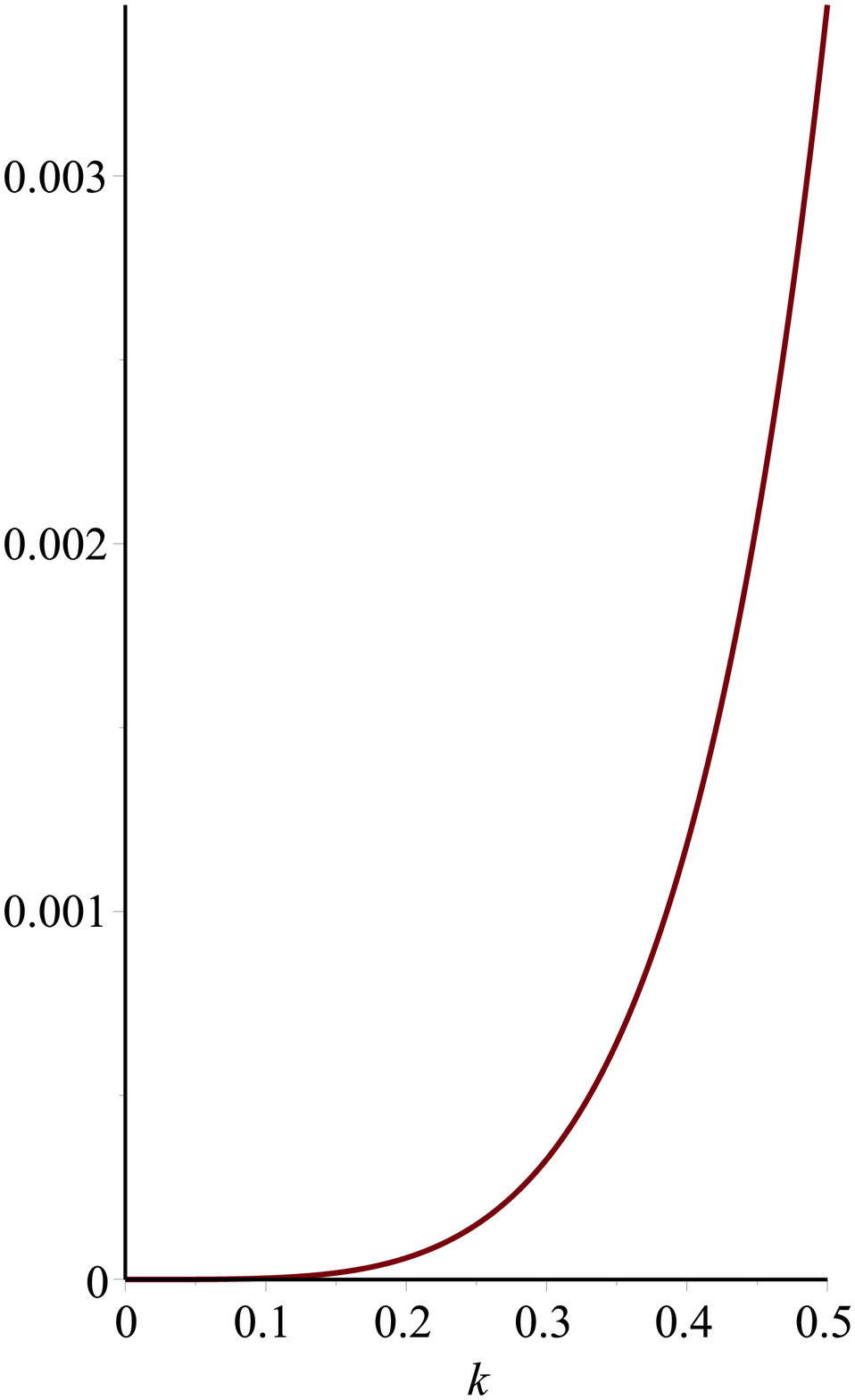}}
	\subfigure{\includegraphics[scale=0.25]{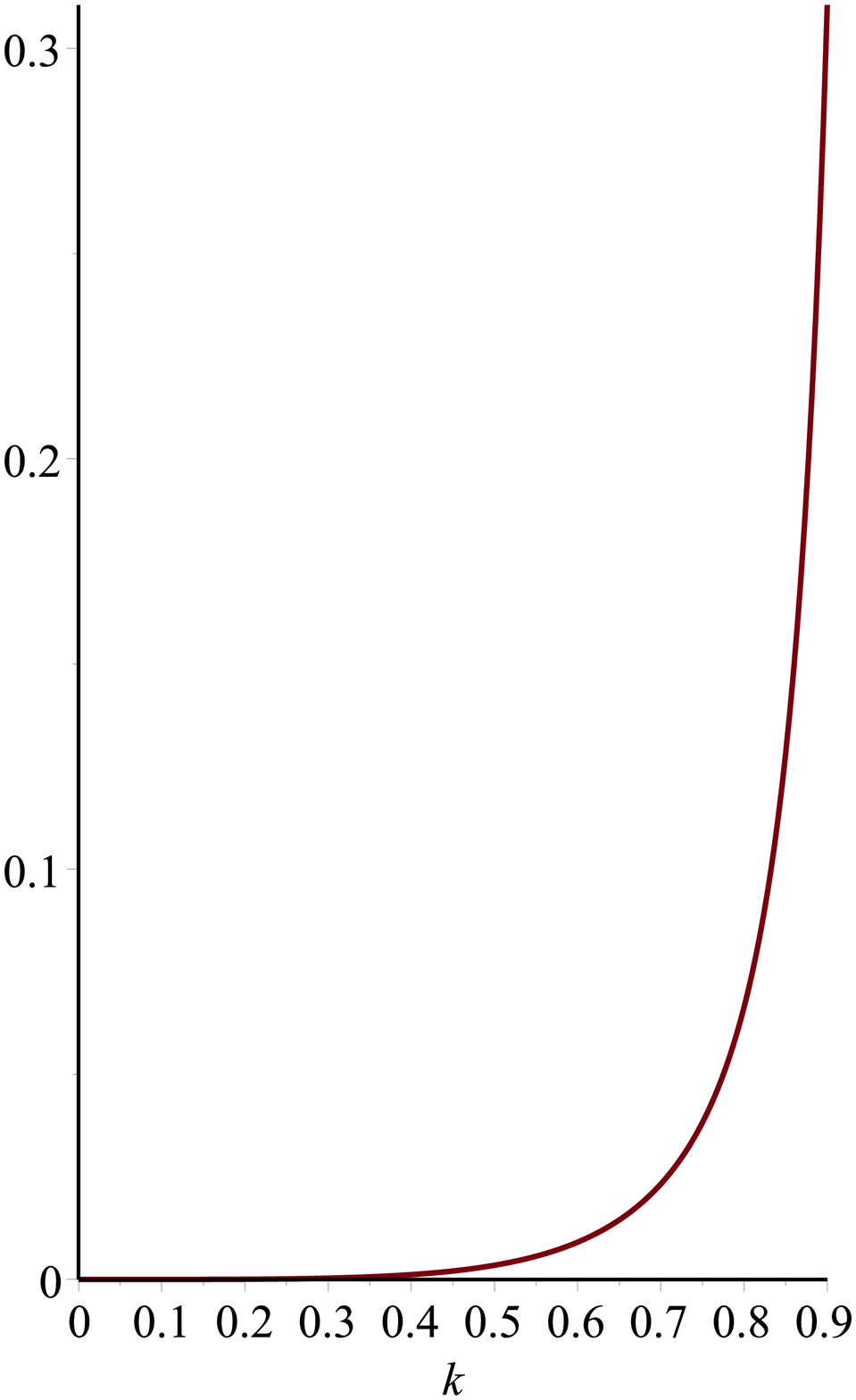}}
	\caption{Graph of $\Phi$ as function of $k$, with $k\in (0,0.2)$, $k\in(0,0.5)$ and $k\in(0,1)$, respectively.}
	\label{mkmkdv1111}
\end{figure}

Next, by following similar ideas as above we can verify that $M_k(\phi_k)>0$ for all $k\in (0,1)$. Since it demands tedious calculation involving elliptic functions, we refrain from write them here and instead only plot the graph of $k\mapsto M_k(\phi_k)$ (see Figure \ref{mkdvMk}).  In addition since $k\mapsto c(k)$ is an increasing function, we see that {\bf (H4)} holds.
\begin{figure}[htb]
	\centering
	\subfigure{\includegraphics[scale=0.25]{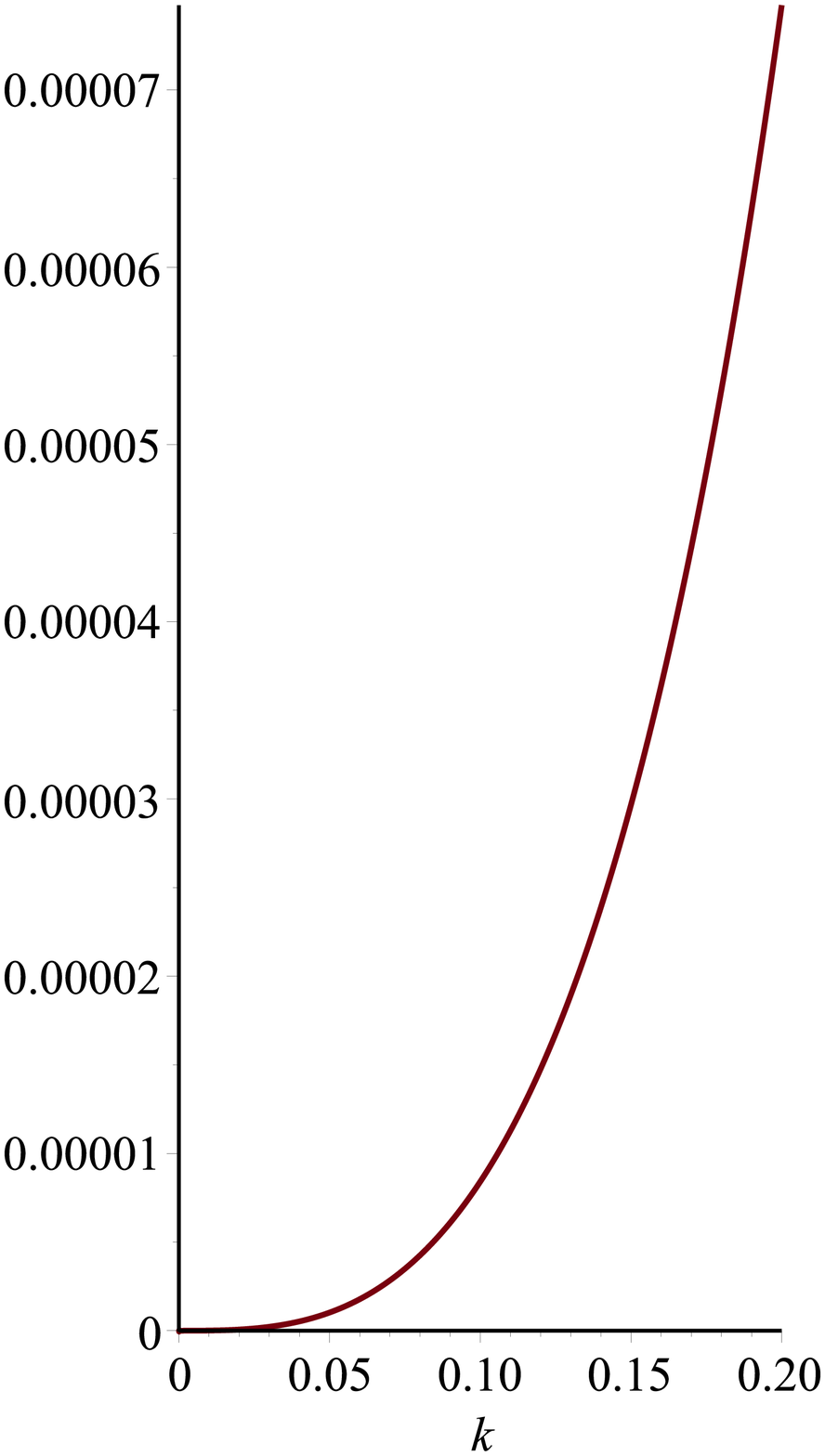}}
	\subfigure{\includegraphics[scale=0.25]{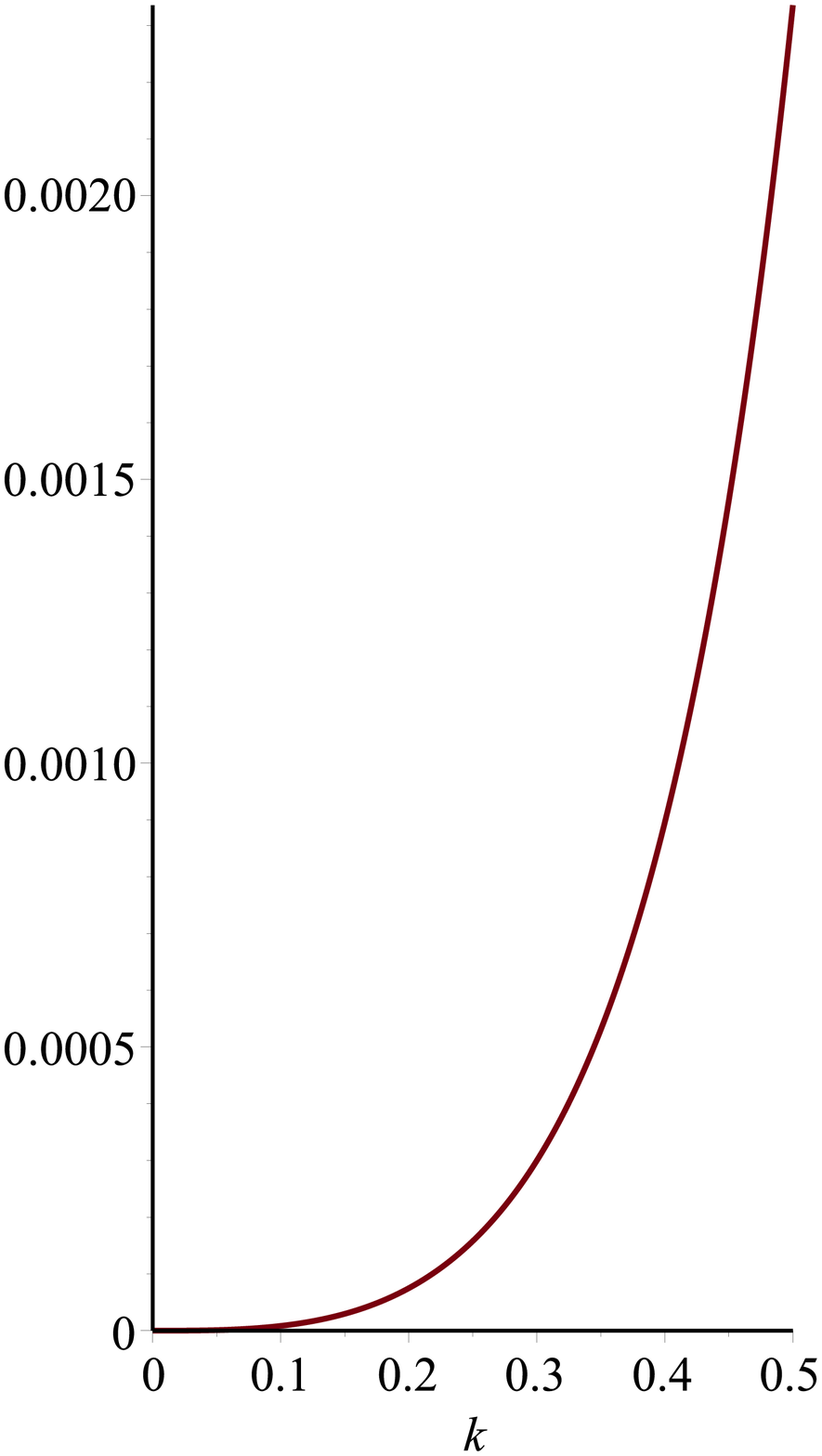}}
	\subfigure{\includegraphics[scale=0.25]{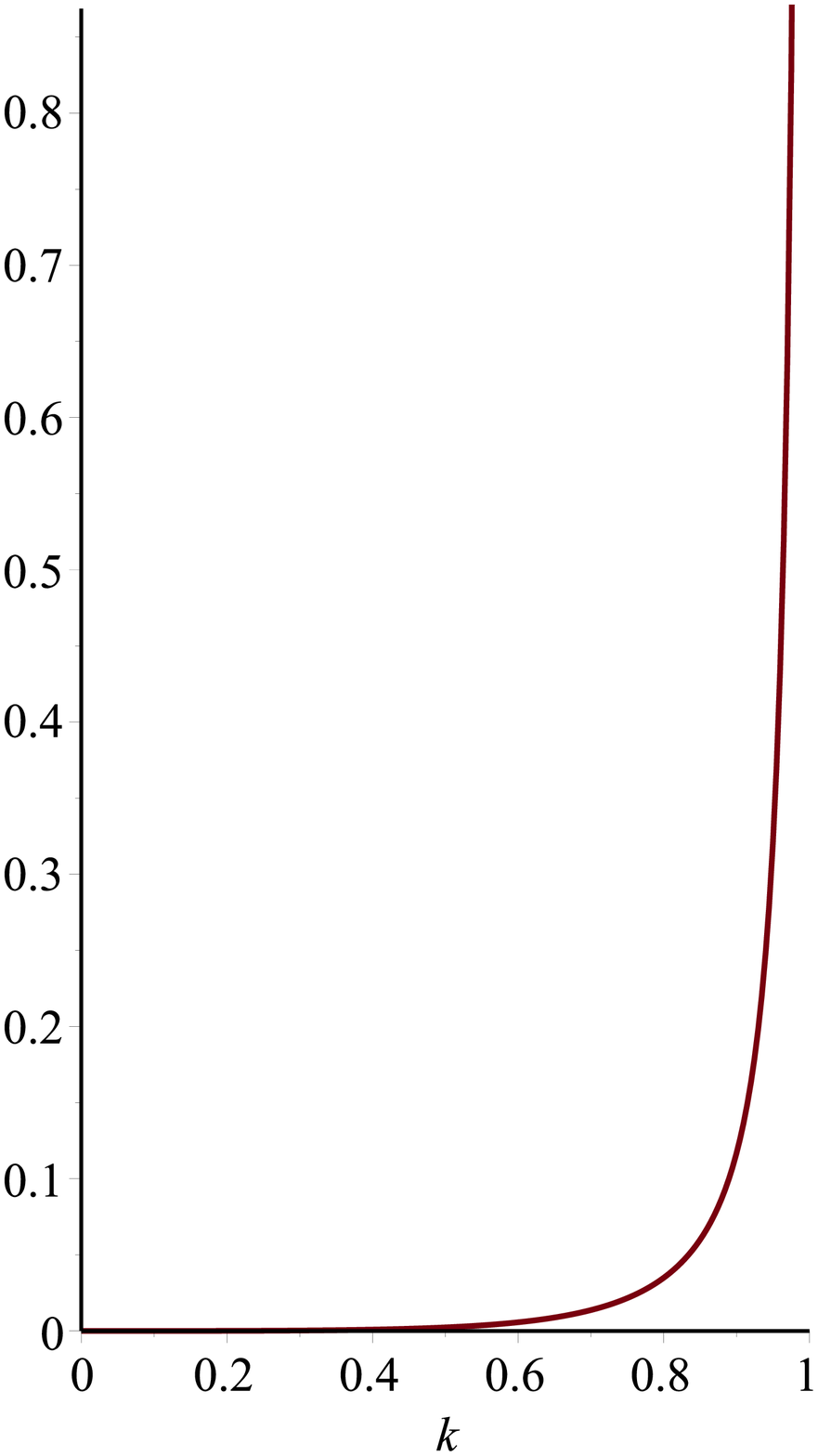}}
	\caption{Graph of $M_k(\phi_k)$ as function of $k$, with $k\in (0,0.2)$, $k\in(0,0.5)$ and $k\in(0,1)$, respectively.}
	\label{mkdvMk}
\end{figure}

 Combining these informations with Theorem \ref{gkdvmaintheorem}, we obtain

\begin{theorem}\label{mkdvstabteo2}
For each $k\in J=(0,1)$, the periodic traveling wave $\phi_k$ given in \eqref{mkdv29}
 is orbitally stable  in $H^{1}_{per}([0,L])$.
\end{theorem}

\subsection{The Gardner equation}

Our purpose in this section is to prove  orbital stability of periodic waves for the Gardner equation 
\begin{equation}
v_t+v_{xxx}+avv_x+bv^2v_x=0,
\label{gd1}
\end{equation}
where $a$ and $b$ are real parameters and $b\neq0$. From the mathematical point of view, Gardner equations is seem as a mixed equation because it contains both KdV and mKdV nonlinearities. On the other hand, Gardner and mKdV equations appear as  models for   the flow of waves in plasma and solid environments. It also appears as a model in quantum fields (see e.g., \cite{konno}, \cite{wadati1} and \cite{wadati2}).

Our strategy  consists in transferring the problem of  orbital stability of periodic waves for \eqref{gd1}  to that for the mKdV equation. Indeed, Gardner equation is close related with the Focusing or Defocusing mKdV equation (F-mKdV or D-mKdV for simplicity)
\begin{equation}
{u}_t+{u}_{xxx}+ \gamma 6{u}^2{u}_x=0,
\label{gd2}
\end{equation}
where $\gamma=\mbox{sgn}(b)$. More precisely, there is  a dipheomorfism $\mathcal{T}$ (between suitable spaces)  that relates solutions of \eqref{gd1} to  solutions of  \eqref{gd2} given by
\begin{equation}
\ds(\mathcal{T}v)(x,t):=\sqrt{\frac{b}{6\gamma}}\left[ v\left(x-\frac{a^2}{4b}t,t\right)+\frac{a}{2b}\right].
\label{gd4}
\end{equation}

Using these transformations, Alejo \cite{alejo1} studied the orbital stability of soliton-like solutions for \eqref{gd1}. The author dealt with the case  $a=6\sigma$ and $b=6$ and proved the orbital stability of the solitary traveling waves
\begin{equation}
\phi_{(\sigma,c_0)}(x-c_\sigma t)=\frac{c_0}{2\sigma+\sqrt{4\sigma^2+c_0}\mbox{cosh}(\sqrt{c_0}(x-c_\sigma t))},
\label{solitonalejo}
\end{equation}
where $c_\sigma=6\sigma^2+c_0$. Here, the constant $c_0$ satisfies $c_0\in(0,\infty)$ if $b>0$, and $c_0\in(0,4\sigma^2)$ if $b<0$. For the stability of $N$-solitons we refer the reader to \cite{alejo2}.

Before proceeding, let us highlight a crucial difference between periodic and solitary wave solutions of \eqref{gd1}. Assume that $v$ is a solution of \eqref{gd1} with $a\neq0$. Then, $\mathcal{T}v$ is a solution of F-mKdV or D-mKdV having the form $\alpha+\beta v$, where $\alpha$ and $\beta$ are real constants with $\alpha\neq0$. In particular, if $v\in C(\R,H^1(\R))$ then $\mathcal{T}v$ does not belong to $C(\R,H^1(\R))$. On the other hand, if  $v\in C(\R,H^1_{per}([0,L]))$ then $\mathcal{T}v$ also belongs to  $C(\R,H^1_{per}([0,L]))$. Therefore, periodic traveling waves solutions of \eqref{gd1} relate in a better way with the periodic traveling waves solution of \eqref{gd2} in the sense that once obtained spatially periodic solutions of \eqref{gd1} we also obtain spatially periodic solutions of \eqref{gd2} (and vice-versa). This is the content of the next lemma.

\begin{lemma} \label{lemapropP}
 Let $\mathcal{T}$ be defined as in  \eqref{gd4}. Then, $\mathcal{T}:C(\R,H^1_{per}([0,L]))\to C(\R,H^1_{per}([0,L]))$ is a diffeomorphism,  whose inverse is given by
 \begin{equation}\label{gd5}
(\mathcal{T}^{-1}u)(x,t):= \sqrt{\frac{6\gamma }{b}}u\left(x+\frac{a^2}{4b}t,t\right)-\frac{a}{2b}.
\end{equation}
In addition, there is a constant $C>0$ such that, for any $u, v \in C(\R,H^1_{per}([0,L]))$,
$$
\ds\rho(u,v)=C\rho(\mathcal{T}u,\mathcal{T}v),
$$
where $\rho$ is defined as in \eqref{rhodef} with $s_2=2$.
\end{lemma}
\begin{proof}
The first part is immediate (see also \cite{kudryashov}). Also, if $u, v \in C(\R,H^1_{per}([0,L]))$ then
$$
\begin{array}{ll}
\rho(\mathcal{T}u,\mathcal{T}v)& \ds = \inf_{r\in\R} \left\|\sqrt{\frac{b}{6\gamma}}\left[ u\left(x-\frac{a^2}{4b}t, t\right)+\frac{a}{2b}-v\left(x-\frac{a^2}{4b}t+ r , t\right)-\frac{a}{2b}\right] \right\|_{H^1_{per}}\\
\\
&=\ds \sqrt{\frac{b}{6\gamma}} \inf_{r\in\R} \left\| u\left(x-\frac{a^2}{4b}t,t\right)-v\left(x-\frac{a^2}{4b}t + r ,t\right) \right\|_{H^1_{per}} \\
\\
&=\ds C \rho(u,v),
\end{array}
$$
where $C=\sqrt{\frac{b}{6\gamma }}$. This completes the proof.
\end{proof}

Note that from Lemma \ref{lemapropP}, up to a constant, the diffeomorphism $\mathcal{T}$ is distance-preserving when measured in the pseudo-metric $\rho$.

Although we are able to obtain the orbital stability of periodic traveling waves in cases $b>0$ and $b<0$, in what follows, we restrict ourselves to the case $b>0$ for simplicity. 
We recall that  periodic traveling waves  of \eqref{gd1}, say,  $v(x,t)=\psi_{(\tilde{c},\tilde{A})}(x-\tilde{c}t)$  are obtained as solutions of 
\begin{equation}
\psi'' -\tilde{c}\psi +\frac{a}{2}\psi^2+\frac{b}{3}\psi^3-\tilde{A}=0
\label{gd6}
\end{equation}
and periodic traveling waves of \eqref{gd2}, say,  $u(x,t)=\phi_{(c,A)}(x-ct)$ are obtained as solutions of 
\begin{equation}
\phi''-c\phi+2\phi^3-A=0.
\label{gd7}
\end{equation}

Since $\mathcal{T}$ takes solutions of \eqref{gd1} to solutions of \eqref{gd2}, $\mathcal{T}$ also takes solutions of \eqref{gd6} to solutions of \eqref{gd7} in the following way: if   $\psi_{(\tilde{c},\tilde{A})}$ is a solution of \eqref{gd6}, then 
$$\phi_{(c,A)}(x):=\mathcal{T}\psi_{(\tilde{c},\tilde{A})}(x)=\sqrt{\frac{b}{6\gamma}}\left[\psi_{(\tilde{c},\tilde{A})}(x)+\frac{a}{2b}\right]$$
 is a solution of \eqref{gd7} with
\begin{equation}
\left\{\begin{array}{l}
\ds c=\tilde{c}+\frac{a^2}{4b},\\
\\
\ds A=\sqrt{\frac{b}{6}}\left(-\tilde{c}\frac{a}{2b}-\frac{a^3}{12b^2}+\tilde{A}\right).
\end{array}\right.
\label{gd8}
\end{equation}
Conversely, if $\phi_{(c,A)}$ is a solution of \eqref{gd7}, then 
$$\psi_{(\tilde{c},\tilde{A})}(x):=\mathcal{T}^{-1}\phi_{(c,A)}(x)=\sqrt{\frac{6}{b}}\phi_{(c,A)}(x)-\frac{a}{2b}$$
 is a solution of \eqref{gd6} with
\begin{equation}
\left\{\begin{array}{l}
\ds \tilde{c}=c-\frac{a^2}{4b},\\
\\
\ds \tilde{A}=\sqrt{\frac{6}{b}}A+ 
c\frac{a}{2b}-\frac{a^3}{24b^2}\, .
\end{array}\right.
\label{gd9}
\end{equation}
Note that \eqref{gd8} and \eqref{gd9} bring and explicit relation among the constants $c$, $\tilde{c}$, $A$ and $\tilde{A}$. This is useful because the wave speed is also known explicitly.

Having in mind the orbital stability results presented for the mKdV equation in Subsection \ref{subsec3.2}, we present two new results of orbital stability for the Gardner equation.

\begin{theorem}
Let $\psi_k$  be defined by
\begin{equation}
\ds\psi_k(\xi)=\frac{2\sqrt{6}K(k)}{\sqrt{b}L}\dn\left(\frac{2K(k)}{L}\xi,k\right)-\frac{a}{2b},
\label{gd1a}
\end{equation}
Then, $\psi_k(\xi)$ is a solution of \eqref{gd6} with
$$\left\{\begin{array}{l}\ds \tilde{c}=c_1(k)=\frac{4K^2(k)}{L^2}(2-k^2)-\frac{a^2}{4b},\\
\\
\ds \tilde{A}=A_1(k)= 
\frac{2aK^2(k)}{bL^2}(2-k^2)-\frac{a^3}{24b^2}\, .
\end{array}\right.
$$
In addition, $v(x,t)=\psi_k(x-c_1(k)t)$ is a periodic traveling solution of \eqref{gd1}, which is orbitally stable in $H^1_{per}([0,L])$.\label{gdteo1}
\end{theorem}
\begin{proof} The first statement follows from the fact that 
\begin{equation}\phi_k(\xi)=\mathcal{T}(\psi_{(c_1(k),A_1(k))}(\xi))=\frac{2K(k)}{L}\dn\left(\frac{2K(k)}{L}\xi,k\right)
\label{gd10}
\end{equation}
is a solution of \eqref{gd7} with (see \eqref{dnmkdv} and \eqref{cmkdv})
$$\left\{\begin{array}{l}\ds{c}(k)=\frac{4K^2(k)}{L^2}(2-k^2),\\
\\
\ds{A}(k)=0.
\end{array}\right.
$$

Now, let $v(t)$ be a solution of \eqref{gd1} with initial data  $v_0$. From Lemma \ref{lemapropP}, it follows that
\begin{equation}
\rho(v(t),\psi_k)=C\rho(\mathcal{T}v(t),\mathcal{T}\psi_k), \,\ \mbox{for all}\,\  t\in\R,
\label{gd11}
\end{equation}
and  $\mathcal{T}v(t)$ is a solution of the mKdV equation with initial data $\mathcal{T}v_0$.
Theorem \ref{mkdvstabteo} implies that the dnoidal solution \eqref{gd10} is orbitally stable by the mKdV flow in $H^1_{per}([0,L])$. Thus, given $\varepsilon_1>0$ there exists $\delta_1>0$ such that if $\rho(\mathcal{T}u_0,\phi_k)<\delta_1$ then
\begin{equation}
\sup_{t\in\R}\rho(\mathcal{T}v(t),\phi_k)<\varepsilon_1.
\label{gd12}
\end{equation}
Next, let  $\varepsilon>0$ be given and choose $\varepsilon_1>0$ such that $\varepsilon_1<\varepsilon/C$. In addition, by choosing $\delta>0$ such that $\delta<C\delta_1$ we see that if $\rho(v_0,\psi_k)<\delta$ then
$$
\rho(\mathcal{T}v_0,\phi_k)=\frac{1}{C}\rho(v_0,\psi_k)<\frac{1}{C}\delta<\delta_1.
$$
Hence, from \eqref{gd11} and \eqref{gd12}, we obtain
$$\sup_{t\in\R}\rho(v(t),\psi_k)=C\sup_{t\in\R}\rho(\mathcal{T}v(t),\phi_k)<\varepsilon,$$
which proves the orbital stability of $\psi_k$.
The proof of the theorem is thus completed.
\end{proof}

\begin{theorem}
Let $\varphi_k$ be defined as
\begin{equation}
\varphi_k(\xi)=\frac{4\sqrt{3}K(k)}{g(k)\sqrt{b}L}\left(\frac{\dn^2\left(\frac{2K(k)}{L}\xi,k\right)}{1+\beta^2\sn^2\left(\frac{2K(k)}{L}\xi,k\right)}\right)-\frac{a}{2b}.
\label{gd1b}
\end{equation}
Then, $\varphi_k(\xi)$ is a solution of \eqref{gd6} with
$$\left\{\begin{array}{l}\ds\tilde{c} = c_2(k)=\frac{16K^2(k)}{L^2}\sqrt{k^4-k^2+1}-\frac{a^2}{4b},\\
\\
\ds\tilde{A}=A_2(k)= 
\frac{8aK^2(k)}{bL^2}\sqrt{k^4-k^2+1} -\frac{32\sqrt{2}K^3(k)\sqrt{2\sqrt{k^4-k^2+1}+2k^2-1}}{3\sqrt{b}L^3\left(\sqrt{k^4-k^2+1}-2k^2+1\right)^{-1}} -\frac{a^3}{24b^2}\, .
\end{array}\right.
$$
In addition, $v(x,t)=\varphi_k(x-c_2(k)t)$ is a solution of \eqref{gd1} orbitally stable in $H^1_{per}([0,L])$.\label{gdteo2}
\end{theorem}
\begin{proof} The first statement now follows from the fact that
\begin{equation}\phi_k(\xi)=\mathcal{T}(\varphi_k(\xi))=\frac{4K(k)}{\sqrt{2}g(k)L}\left(\frac{\dn^2\left(\frac{2K(k)}{L}\xi,k\right)}{1+\beta^2\sn^2\left(\frac{2K(k)}{L}\xi,k\right)}\right)\label{tp}\end{equation}
is a solution of \eqref{gd7} with
$$\left\{\begin{array}{l}\ds c(k)=\frac{16K^2(k)}{L^2}\sqrt{k^4-k^2+1},\\
\\
\ds{A}(k)= 
 -\frac{32K^3(k)}{3\sqrt{3}L^3}\left(\sqrt{k^4-k^2+1}-2k^2+1\right)\sqrt{2\sqrt{k^4-k^2+1}+2k^2-1}.
\end{array}\right.
$$
In fact, the periodic travelling wave $\phi_k$ in \eqref{tp} is exactly the solution of the mKdV which we proved to be orbitally stable in $H^1_{per}([0,L])$ (see Theorem \ref{mkdvstabteo2}).
Therefore, the rest of the proof follows exactly the same arguments as in the previous theorem. 
\end{proof}

\subsection{The ILW equation} 

Next we consider the Intermediate Long Wave (ILW) equation
\begin{equation}\label{ilw}
u_t-\mathcal{M}_\delta u+\partial_x(u^2)=0,
\end{equation}
where $u=u(x,t)$ is $L$-periodic in the spatial variable and the linear operator $\mathcal{M}_\delta$ is defined via Fourier transform by
$$
\widehat{\mathcal{M}_\delta u}(m)=\left( \frac{2\pi m}{L}{\rm coth}\left( \frac{2\pi m \delta}{L}\right) -\frac{1}{\delta} \right)\widehat{u}(m), \qquad \delta>0.
$$
Equation \eqref{ilw} is derived as a model equation for long, weakly nonlinear internal gravity waves in a stratified fluid of finite depth. The parameter $\delta>0$ is close related with the depth of the fluid (see e.g., \cite{kkb}).

The operator $\mathcal{M}=\mathcal{M}_\delta$ satisfies our assumptions with $\gamma=0$ and $s_1=s_2=1$. The authors in \cite{acn} shown that \eqref{ilw} has  traveling wave solution of the form $u(x,t)=\phi_k(x-c(k)t)$, with $k\in J:=(0,k_1)\subset(0,1)$,
$$
c(k):=\frac{1}{\delta}-\frac{8\pi K(k)}{L^2K(k')}-\frac{4K(k)}{L}\left[Z(\alpha,k')+\frac{\cn(\alpha,k')\dn(\alpha,k')}{\sn(\alpha,k')}\right]>0, \quad \alpha=\frac{4\delta K(k)}{L},
$$
$$
A(k)=\frac{1}{L}\int_0^L\varphi_k^2(x)dx,
$$
and
$$
\phi_k(x)=-\frac{4K(k)}{L}Z(\delta y,k)-\frac{4\delta \pi}{L^2}\frac{K(k)}{K(k')}+\frac{4K(k)}{L}\frac{\dn^2(y,k)\cn(\delta y,k')\sn(\delta y,k')\dn(\delta y,k')}{1-\dn^2(y,k)\sn^2(\delta y,k)}.
$$ 
Here, $y=\frac{2K(k)x}{L}$, $k'=\sqrt{1-k^2}$ and $Z(x,k)$ stands for the Jacobian Zeta function defined by
$$
Z(x,k)=\int_0^x\left( \dn(s,k)-\frac{E(k)}{K(k)}\right) ds.
$$
This shows that {\bf (H0)} holds. Also, by using the total positivity theorem as described in the beginning of this section, it was shown that the linearized operator $\Lk=\mathcal{M}_\delta+c-2\phi_k$ satisfies {\bf (H1)-(H2)}. Assumptions {\bf (H3)} and {\bf (H4)} were also checked in \cite{acn} taking the advantage of the explicit form of the solutions $\phi_k$. Consequently, one deduces the orbital stability of the periodic wave $\phi_k$. For details we refer the reader to \cite[Section 6]{acn}.

\subsection{The Schamel equation}

This subsection is devoted to the so called Schamel equation
\begin{equation}
\displaystyle
u_t+\partial_x(u_{xx}+|u|^{3/2})=0,
\label{schamel}
\end{equation}
which governs the behaviour of weakly nonlinear ion-acoustic solitons that are modified by
the presence of trapped electrons. Such equations was first derived  by Schamel in \cite{sch}, \cite{sch1}. The existence of periodic solutions can be obtained by using the quadrature method. In particular, in \cite{cnp} the authors established that
\begin{eqnarray}\label{stab.26}
\phi_{k}(\xi)&=&\displaystyle\frac{6400 K(k)^4}{9 L^4}\displaystyle\left[1-2k^2+\displaystyle\sqrt{1-k^2+k^4}+3k^2\cn^2\displaystyle\left(\displaystyle\frac{2K(k)\xi}{L};k\right)\right]^2\nonumber\\
\\
&=&\displaystyle\frac{6400 K(k)^4}{9 L^4}\displaystyle\left[k^2-2+\displaystyle\sqrt{1-k^2+k^4}+3\dn^2\displaystyle\left(\displaystyle\frac{2K(k)\xi}{L};k\right)\right]^2,\nonumber
\end{eqnarray}
is an $L$-periodic traveling wave solution of \eqref{schamel}, where for each $k\in(0,1)=:J$,
\begin{eqnarray}\label{stab.21}\
c=\displaystyle\frac{64 K(k)^2}{L^2}\displaystyle\sqrt{1-k^2+k^4},
\end{eqnarray}
and
\begin{eqnarray}\label{stab.25}
A=\displaystyle\frac{204800K(k)^6}{27 L^6} \displaystyle\left[-2k^6+3k^4+3k^2-2-2(1-k^2+k^4)^{\frac{3}{2}}\right].
\end{eqnarray}
This shows that {\bf (H0)} holds with $J=(0,1)$. Properties {\bf (H1)-(H2)} can be established by studying the periodic eigenvalue problem
\begin{equation}\label{lame1}
\begin{cases}
 -{y}''+\displaystyle\left(c(k)-\displaystyle\frac{3}{2}\phi_{k}^{\frac{1}{2}}\right){y}=\lambda y, \\
 y(0)=y(L),\;\;y'(0)=y'(L),
\end{cases}
\end{equation}
or using the same technique as in Subsection \ref{dssub} (see \cite{cnp} and \cite{depas} for details). On the other hand, {\bf (H3)-(H4)} can be checked by using the explicit form of the quantities involved. Consequently, Theorem \ref{gkdvmaintheorem} may also be applied to obtain the orbital stability of the periodic traveling waves \eqref{stab.26}. We refer the reader to \cite{cnp} for the details.

\section{Extension to regularized equations}\label{gbbmsec}

In this section, we extend the theory developed in Section \ref{section2} to regularized equations of the form
\begin{equation}\label{regeq}
u_t+\mtm u_t+\partial_x(u+f(u))=0,
\end{equation}
where $\mtm$ and $f$ satisfy the  assumptions posed in the introduction.
Equations of this form arise as models of wave propagation in a variety of physical contexts.

Traveling waves solutions for \eqref{regeq} are also special solutions having the form $u(x,t)= \phi(x-ct)$. By replacing this form of wave in \eqref{regeq}, $\phi$ must solve 
\begin{equation}\label{regeqsol}
c\mtm \phi+(c-1)\phi-f(\phi)+A=0,
\end{equation}
where again $A$ is an integration constant.  

It is well-known that \eqref{regeq} has three conserved quantities, namely,
\begin{equation}\label{energyreg}
E(u)=\frac{1}{2}\int_0^L (u\mtm u-2F(u))dx, \quad \mbox{with} \quad F(u)=\int_0^uf(s)ds,
\end{equation}
\begin{equation}\label{massreg}
Q(u)=\frac{1}{2}\int_0^L (u^2+u\mtm u)dx,
\end{equation}
and
\begin{equation}\label{mmreg}
V(u)=\int_0^L u\,dx.
\end{equation}
Solutions of \eqref{regeqsol} are now critical points of the functional
$$
E+(c-1)Q+AV.
$$
With this in hand, our assumptions read as follows.

\begin{itemize}
\item[{\bf (P0)}] There are an interval $J\subset\R$, $C^1$-functions $k\in J\mapsto c=c(k)$ and $k\in J\mapsto A=A(k)$, and a nontrivial smooth curve of $L$-periodic solutions for \eqref{regeqsol}, $ k\in J  \mapsto \phi_k:=\phi_{(c(k),A(k))} \in H^{s_2}_{per}([0,L])$ with $c=c(k)>1$.
\end{itemize}

\begin{itemize}
\item[{\bf(P1)}] The linearized operator $\mathcal{L}_k:=c\mtm+(c-1)-f'(\phi_k)$, defined on a dense subspace of $L^2_{per}([0,L])$, has a unique negative  eigenvalue, which is simple.
\item[{\bf(P2)}] Zero is a simple eigenvalue of $\mathcal{L}_k$ with associated  eigenfunction $\phi_k'$.
\end{itemize}

\begin{itemize}
\item[{\bf(P3)}] The quantity $\Phi$ defined by $\Phi:=\left\langle\mathcal{ L}_k\left(\frac{\partial \phi_k}{\partial k}\right),\frac{\partial \phi_k}{\partial k}\right\rangle$
 is negative.
 \item[{\bf(P4)}] It holds $M_k(\phi_k)\neq -\dfrac{\partial c}{\partial k}Q(\phi_k)$, where $M_k(u):=\ds\frac{\partial c}{\partial k}Q(u)+\frac{ \partial A}{\partial k} V(u)$.
\end{itemize}

The theory developed in Section \ref{section2} extends mutatis mutandis to the present situation and we can prove the following.

\begin{theorem}[Orbital stability]\label{regtheorem}
Under assumptions {\bf (P0)-(P4)},
for each $k\in J$, the periodic traveling wave $\phi_k$
 is orbitally stable by the flow of \eqref{regeq} in $H^{s_2/2}_{per}([0,L])$, that is,
 for any $\varepsilon>0$, there exists $\delta>0$ such that if $u_0\in H^{s_2/2}_{per}([0,L])$ satisfies
$$
\|u_0-\phi\|_{H^{s_2/2}_{per}}<\delta,
$$
then the solution $u(t)$ of \eqref{regeq}, with initial data $u_0$,  satisfies
$$
\sup_{t\in\R}\inf_{r\in\R}\|u(t)-\phi_k(\cdot+r)\|_{H^{s_2/2}_{per}}<\varepsilon.
$$
\end{theorem}

\begin{remark}
Note that in the proof of Lemma \ref{lemmaconv1} the sign of the coefficients in $$f(\alpha)=a\alpha^2+b\alpha$$ is not relevant. So, the sign of the quantity $Q_k(\phi_k)$ does  not change the arguments in the proof of Lemma \ref{lemmaconv1}.  As a result, the proof of Theorem \ref{regtheorem} is similar to that of Theorem \ref{gkdvmaintheorem}.
\end{remark}

\subsection{The modified BBM equation}
This section is devoted to study the orbital stability of periodic waves for the modified BBM equation
\begin{equation}
\label{mbbmequation}
u_t-u_{xxt}+(u+u^3)_x=0.
\end{equation}
At least from the mathematical point of view, \eqref{mbbmequation} can be viewed as a regularized version of the mKdV equation (see also \cite{bbm}).
 In such a case, the traveling wave $\phi$ must be a solution of the equation
\begin{equation}
\label{mbbmedo}
-c\phi''+(c-1)\phi-\phi^3+A=0,
\end{equation} 
If it is assumed that $A=0$ then \eqref{mbbmedo} has a dnoidal-type solution of the form
\begin{equation}\label{bbmdn}
\phi_c(\xi)=\eta\,\dn \left(\frac{\eta x}{\sqrt{2c}}\right),
\end{equation}
where $\eta$ is a real constant. The stability of the solution \eqref{bbmdn} in the energy space $H^1_{per}([0,L])$ was studied in \cite{abs}. In particular, the authors have shown the orbital stability of $\phi_c$. The method used to obtain the spectral properties was based on the total positivity theory; whereas the techniques to prove the stability itself was based on the classical method.

Our goal here is to study a solution of \eqref{mbbmedo} when $A\neq0$. Indeed,
 by applying the quadrature method and using formula 257.00 in \cite{friedman} we obtain a solution in terms of the elliptic functions given by
\begin{equation}\label{wavebbm}
\phi_k(x)=\frac{4\sqrt{c}K(k)\dn^2\left(\frac{2K(k)}{L}x,k\right)}{g(k)L\left(1+\beta^2\sn^2\left(\frac{2K(k)}{L}x,k\right)\right)},
\end{equation}
where $\beta^2=\sqrt{k^4-k^2+1}+k^2-1$ and $g(k)=\sqrt{\sqrt{k^4-k^2+1}-k^2+\frac{1}{2}}$. Here, both $c$ and $A$ are also considered as functions of $k$. More precisely, 
$$
c(k)=\frac{L^2}{L^2-16K^2(k)\sqrt{k^4-k^2+1}}
$$
and
$$
A(k)=\frac{16c\sqrt{c}K^3(k)(r^4(k)-9)}{3\sqrt{6}L^3r(k)},$$
where $r(k)=\sqrt{2\sqrt{k^4-k^2+1}+2k^2-1}$. Note that since the function $k\in(0,1)\mapsto16K^2(k)\sqrt{k^4-k^2+1}$ is strictly increasing, $c(k)$ has a unique singular point, generally close to 1, which we shall call $k_L$.  Figure \ref{figmbbm} illustrates the behaviour of the functions $c(k)$ and $A(k)$ with $k\in(0,k_L)$.
\begin{figure}[htb]
	\centering
	\subfigure{\includegraphics[scale=0.3]{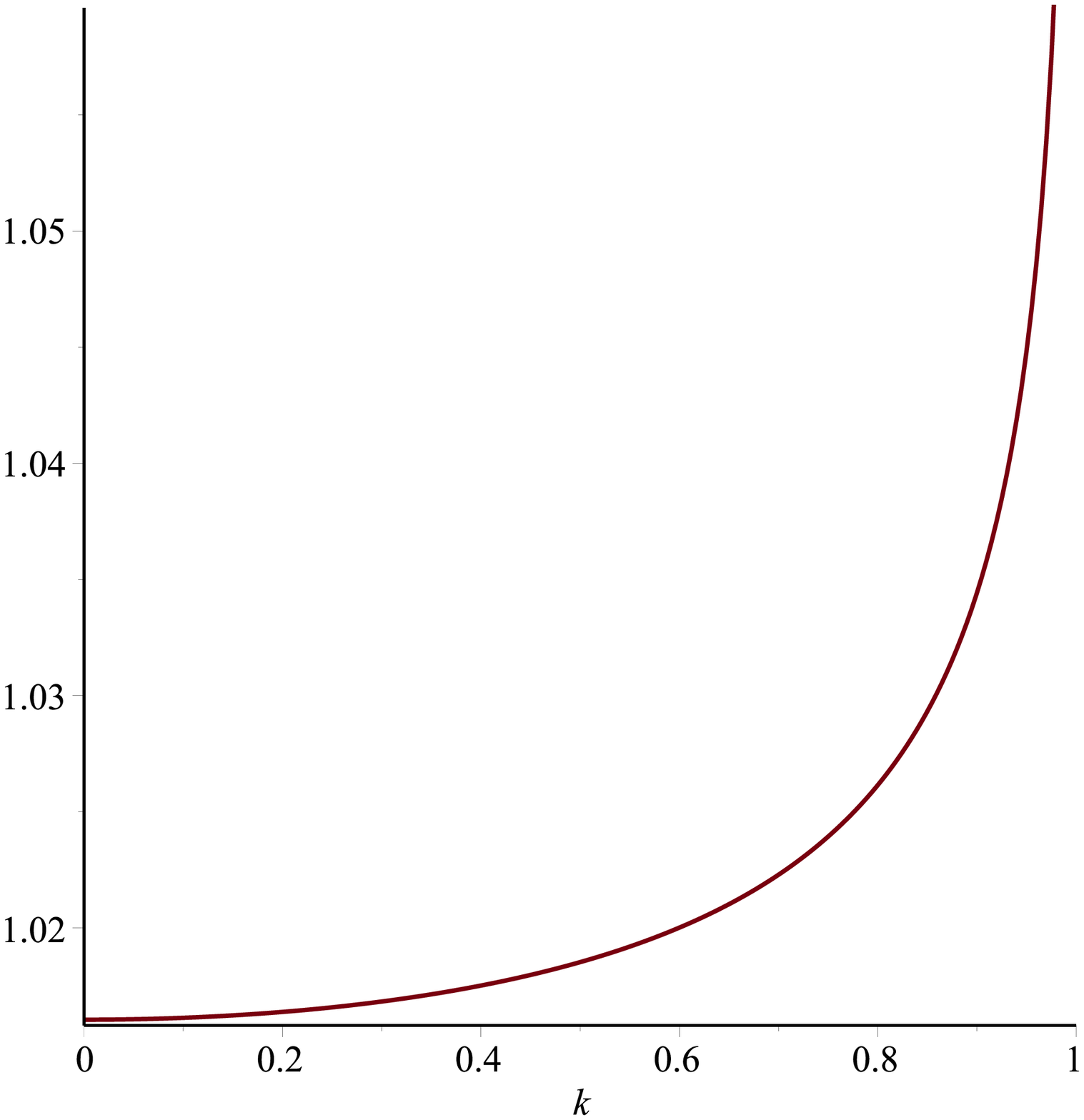}}
	\qquad
	\subfigure{\includegraphics[scale=0.3]{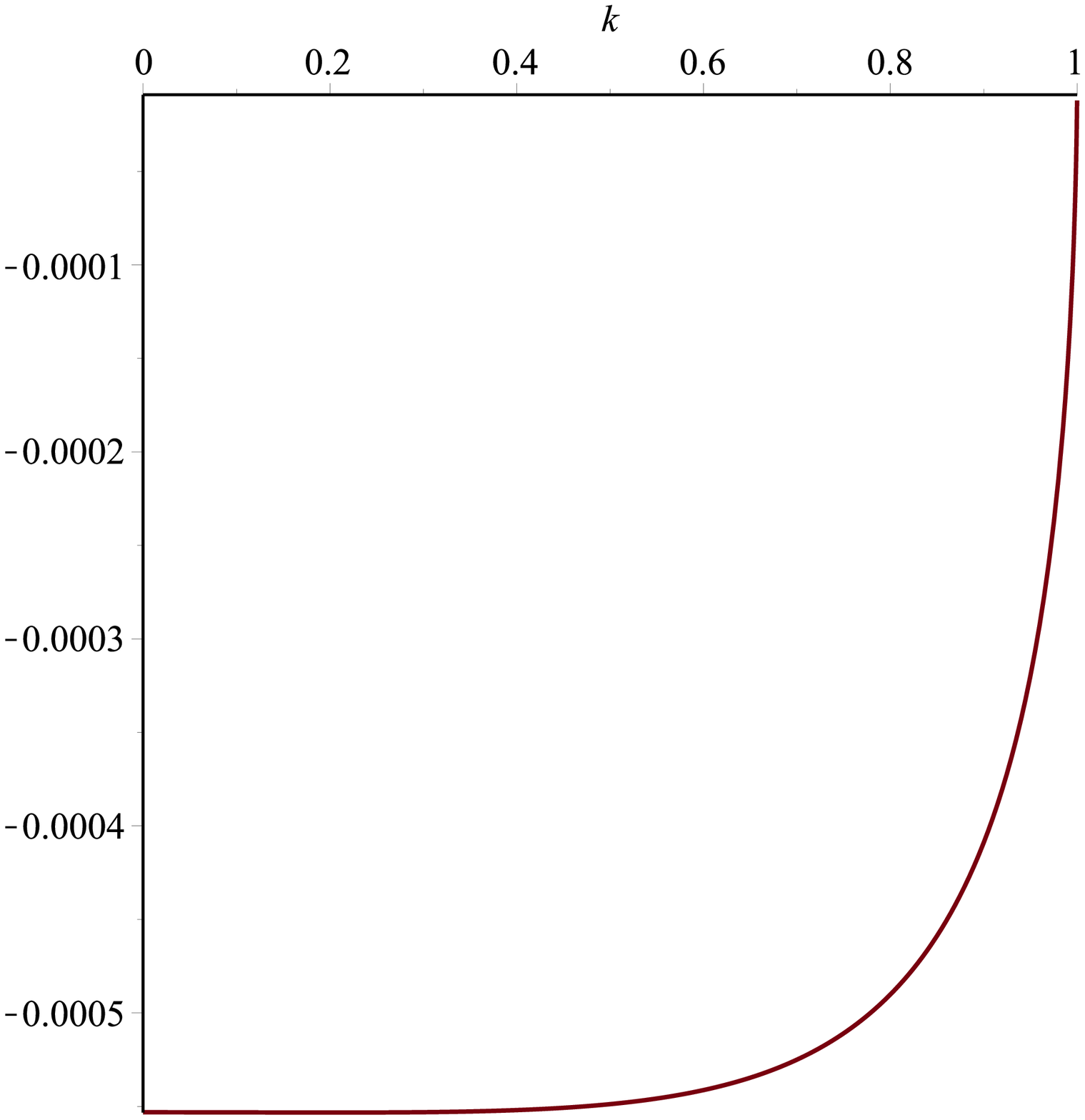}}
	\caption{Left: Graph of $k\in(0,k_L)\mapsto c(k)$. Right: Graph of $k\in(0,k_L)\mapsto A(k)$. In both cases, $L=50$ .}
	\label{figmbbm}
\end{figure}

Notice the function $c(k)$ is increasing on $(0,k_L)$ and $c(k)>1$. Thus, we must have
$$\frac{L^2}{L^2-4\pi^2}>1.$$
This inequality implies the period $L$ of the profile $\phi_k$ in \eqref{wavebbm} must satisfy $L>2\pi$. As a consequence, the condition ${\bf (P0)}$ is fulfilled, with $J:=(0,k_L)$.

In what follows we check conditions {\bf (P1)-(P2)}. We will use a similar analysis as in Subsection \ref{subsec3.2}.  Since, $c(k)>1$ it suffices to check such conditions for
	$$ 
	\tilde{\Lk}=-\partial_x^2+\frac{(c(k)-1)}{c(k)}- \frac{3}{c(k)}\phi_k^2.
	$$
Thus, because $c$ and $\phi_k$ depends smoothly on $k$, $\tilde{\Lk}$ is isonertial, that is, the inertial index $(n,z)$ does not depend on $k$. Therefore, as before it suffices to fix $k\in(0,k_L)$ and $L>0$. Let us fix $k_0:=0.5$ and $L_0=50$. Since $\phi_{k_0}'$ has two zeros in the interval $[0,L_0)$ and 	$\tilde{\Lk}(\phi_k')= 0,$
	we obtain that zero is the second or the third eigenvalue of $\mathcal{L}_{k_0}$. In order to show that zero is indeed the second one, it suffices to prove that
	\begin{equation}
	\theta:=\frac{y'(L_0)}{\phi_{k_0}''(0)}<0,
	\label{formulathetambbm}
	\end{equation}
	where $y$ is the unique solution of
	\begin{equation}
	\left\{\begin{array}{l}
	\ds-y''+\frac{1}{c(k_0)}\left[(c(k_0)-1)-3\phi_{k_0}^{2}\right]y=0,\\
	y(0)=-\frac{1}{\phi_{k_0}''(0)},\\
	y'(0)=0.
	\end{array}\right.
	\label{pvialoisiombbm}
	\end{equation}
	The constant $\theta$ can be determined by solving numerically \eqref{pvialoisiombbm}. In particular, we deduce that 
	$$\theta\cong-8.516957300\times 10^5.$$
Thus, {\bf (P1)-(P2)} are checked. 
	Table \ref{table1} illustrates some values of $\theta$ if we fix $k_0=0.4$ and choose different values of $L$. Note that $\theta$ is always negative and increases, in absolute value, with $L$.

\begin{table}[htb]
\centering
	\caption{Values of $\theta$ with  $k_0=0.4$.}
	\begin{tabular}{|c|c|c|c|c|} 
		\hline 
		$L=10$ & $L=20$ & $L=200$ & $L=1000$ & $L=100000$ \\ 
		\hline 
		$\theta\cong -166.08$ & $\theta\cong -7976.14$ & $\theta\cong -8.85 \times 10^8$ & $\theta\cong -2.76 \times 10^{12}$ & $\theta\cong -2.73 \times 10^{22}$ \\ 
		\hline 
	\end{tabular}
\label{table1}
\end{table}

We now proceed to check {\bf (P3)}. By deriving  equation (\ref{mbbmedo}), with respect to $k$, we obtain
$$\mathcal{L}_k \left(\frac{\partial \phi_k}{\partial k}\right)=-\frac{\partial c}{\partial k}\left(-\phi_k''+\phi_k\right)-\frac{\partial A}{\partial k}.$$
Therefore, with a similar analysis as in Subsection \ref{subsec3.2}, we can write $\Phi$ as
$$\begin{array}{rl}
\Phi=\ds&-\ds  \frac{\partial c}{\partial k} \frac{\partial }{\partial k}\left(\frac{1}{2}\int_0^L\left(\left(\phi_k'\right)^2+ \phi_k^2\right)dx\right) - \frac{\partial A}{\partial k}\frac{\partial}{\partial k}\int_0^L\phi_k dx\\
\\
=&\ds- \frac{\partial c}{\partial k} \frac{\partial }{\partial k}Q(\phi_k) - \frac{\partial A}{\partial k}\frac{\partial}{\partial k}V(\phi_k).
\end{array}
$$
 
We now can compute numerically the values of  $Q(\phi_k)$ and $V(\phi_k)$ and  obtain $\Phi<0$. An illustration of the behaviour of $\Phi$ as function of $k$ is given in Figure \ref{mbbmP3}.
\begin{figure}[htb]
	\centering
	\subfigure{\includegraphics[scale=0.3]{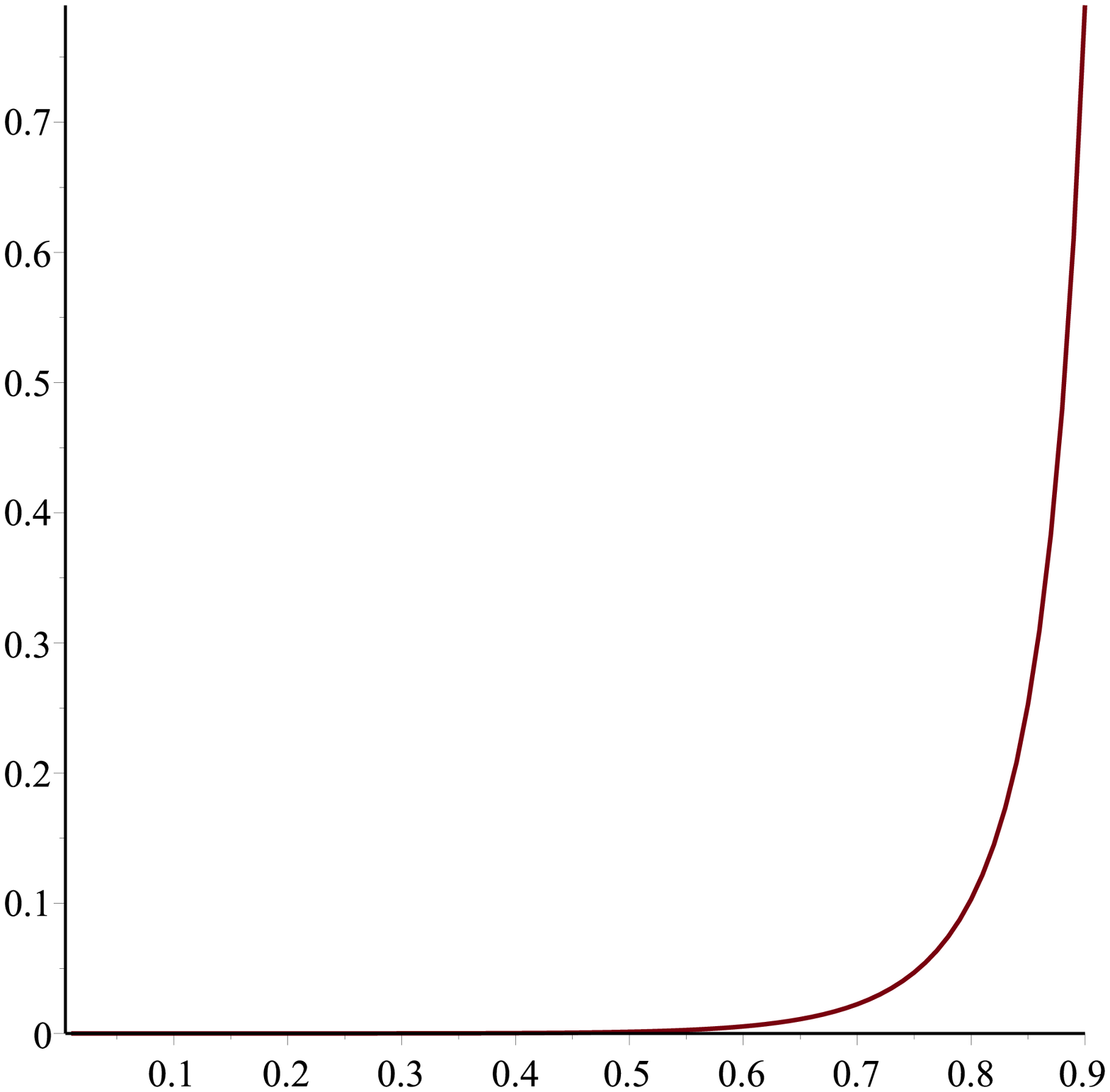}}
	\qquad
	\subfigure{\includegraphics[scale=0.3]{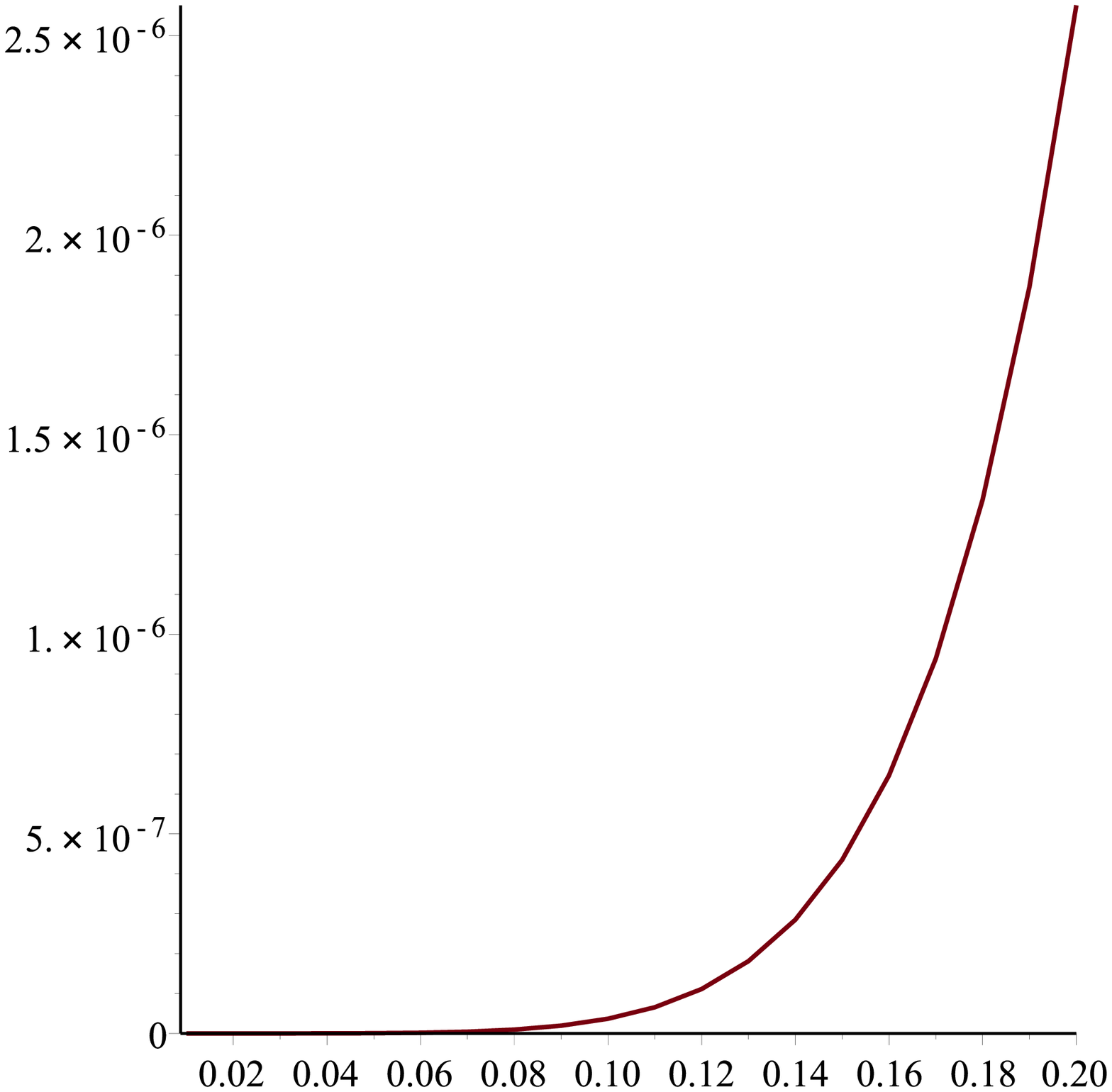}}
	\caption{Left: Graph of $\Phi$ as function of $k$, with $k\in(0,k_L)$. Right: Graph of $\Phi$ as function of $k$, with $k\in(0,0.2)$. In both cases, $L=30$.}
	\label{mbbmP3}
\end{figure}

The expressions of $\phi_k$, $Q$ and $V$ also allow us to  check, numerically, that  $\Psi:=M_k(\phi_k)+\frac{\partial c}{\partial k}Q(\phi_k)\neq0$ and condition {\bf (P4)} hods. See Figure \ref{mbbmP4}.

\begin{figure}[htb]	
	\centering
	\subfigure{\includegraphics[scale=0.3]{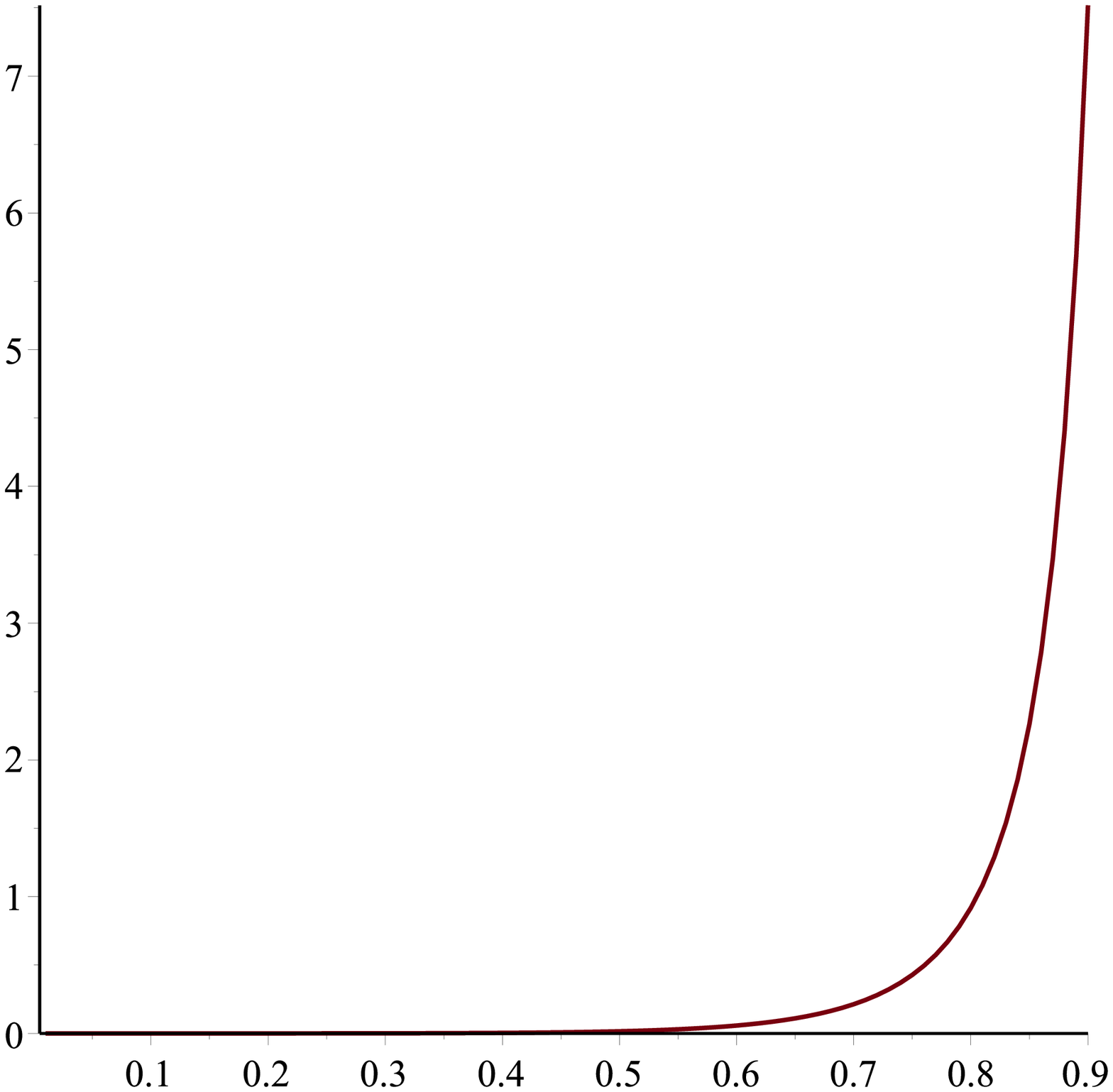}}
	\qquad
	\subfigure{\includegraphics[scale=0.3]{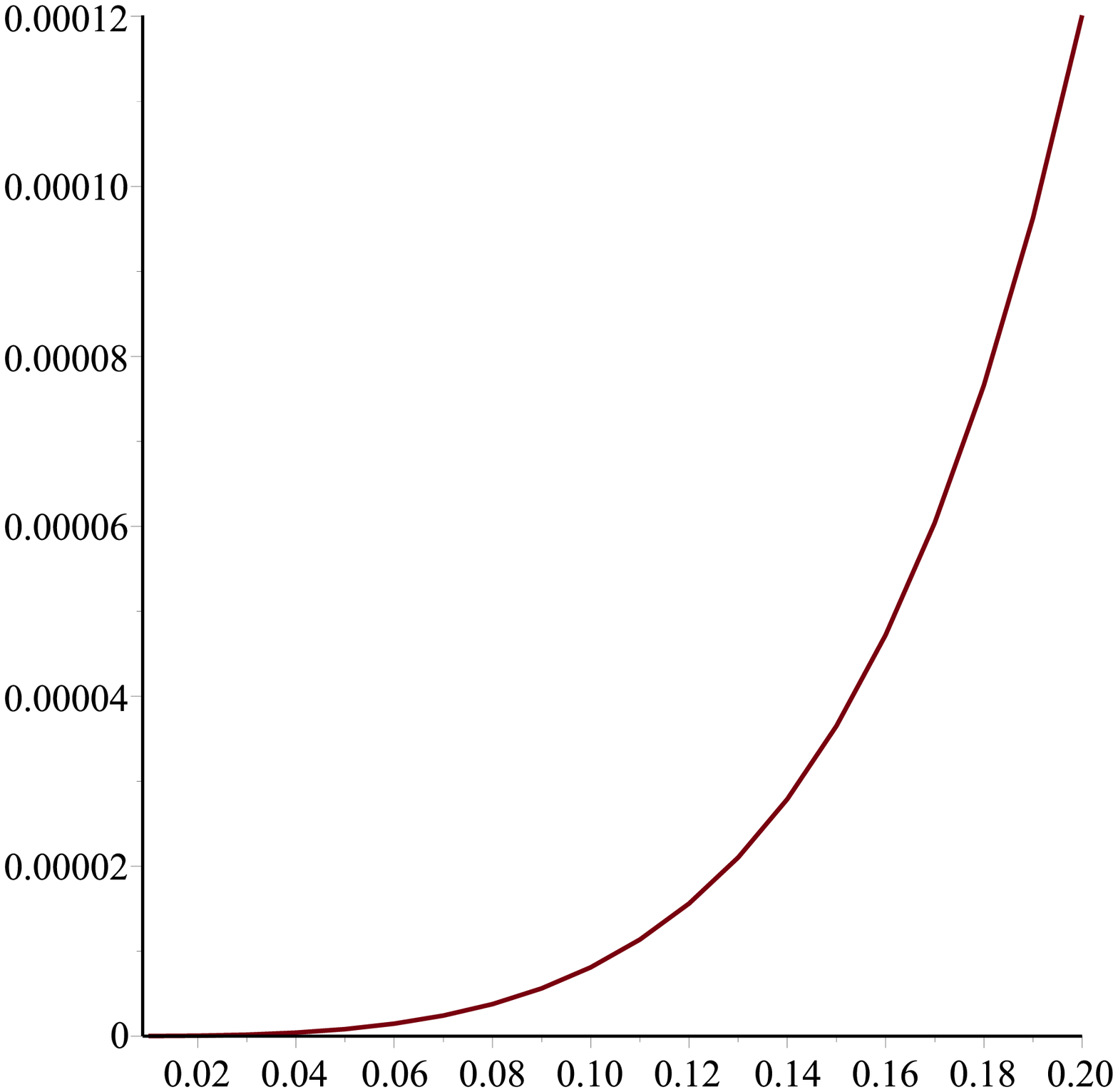}}
	\caption{Left: Graph of $\Psi$ for  $k\in(0,k_L)$. Right: Graph of $\Psi$ for  $k\in(0,0.2)$. In both cases, $L=30$.}
	\label{mbbmP4}
	
\end{figure}

Table \ref{table2} below also gives the values of $\Phi$ and $\Psi$ for different values of $k$. As an application of Theorem \ref{regtheorem} we then have proved the following.

\begin{theorem}\label{mbbmstabteo}
	For each $k\in J=(0,k_L)$, the periodic traveling wave $\phi_k$ given in \eqref{wavebbm}
	is orbitally stable in $H^{1}_{per}([0,L])$.
\end{theorem}

\begin{table}[htb]
	\centering
	\caption{$\Phi$ and $\Psi$ for some values of $k$.}
	\begin{tabular}{|c|c|c|} 
		\hline 
	$k$ & $\Phi$ & $\Psi$ \\ 
		\hline 
0.1	& $3.675157856\times 10^{-8}$& 0.00009093638233\\
		\hline 
0.2 &0.000002575957430 &0.0007877593065 \\
 \hline
0.3		& 0.00003422439640 & 0.003061968119\\
		\hline
0.4		 &0.0002398363306 & 0.009008424446\\
		 \hline
0.5		& 0.001228499118 & 0.02397754841\\
		\hline
0.6		 & 0.005375083538&0.06355524106 \\
		 \hline
0.7		& 0.02241081146 &0.1814545325 \\
		\hline
0.8		 &0.1029912842 & 0.6356553017\\
		 \hline
0.9		& 0.7898496312 & 4.353282492 \\
		\hline
	\end{tabular}
\label{table2}
\end{table}

\subsection{The regularized Schamel equation}

Here we consider the so called regularized Schamel equation
\begin{equation}
u_t-u_{xxt}+\partial_x(u+|u|^{3/2})=0,
\label{bbm}
\end{equation}
which can be viewed as a regularized version of \eqref{schamel} in much the same way that the BBM equation can be viewed as a regularized version of the KdV equation.

Periodic traveling waves solutions for \eqref{bbm} may be obtained in view of the quadrature method. Indeed, in \cite{depas} it was shown that for $L>4\pi$ there exists $k_L\in(0,1)$ such that for $k\in(0,k_L)$ the function
\begin{equation}
\phi_{k}(x)=\left[\!\frac{5}{12}\!\left(\!\frac{L^2\!-\!64(2k^2\!-\!1)K^2(k)}{\tilde{m}(k)}-1\! \right)\!+\!\frac{80k^2K^2(k)}{\tilde{m}(k)}\mathrm{CN}^2 \left(\!\frac{2K(k)}{L} x,k\!\right)\! \right]^2,
\label{phikL}
\end{equation}
where $\tilde{m}(k)=L^2-64K^2(k) \sqrt{k^4-k^2+1}$, is a traveling wave for \eqref{bbm} with
\begin{equation}
c=\frac{L^2}{L^2-64 K^2(k)\sqrt{k^4-k^2+1}}
\label{ck}
\end{equation}
and
\begin{equation}
A=\frac{-204800K^6(k)}{27\tilde{m}^3(k)}\left[\left( \sqrt{k^4-k^2+1}-(2k^2-1)\right)^2\left(2\sqrt{k^4-k^2+1}+(2k^2-1)\right)\right].
\label{AkL}
\end{equation}
The hypothesis {\bf (P0)} is then fulfilled with $J:=(0,k_L)$. Properties {\bf (P1)-(P2)} can be obtained by studying the periodic eigenvalue problem  associated with the operator $\mathcal{L}_k=-c\partial_x^2+(c-1)-\frac{3}{2}\phi_k^{1/2}$, which in turn is equivalent to study the eigenvalue problem associated with a Lam\'e type equation, namely,
\begin{equation}\left\{\begin{array}{l}
\ds\Lambda''(x)+\left[h-5\cdot 6\cdot k^2 \sn^2\left(x,k\right)\right]\Lambda(x)=0,\\
\Lambda(0)=\Lambda(2K(k)), \quad \Lambda'(0)=\Lambda'( 2K(k)),
\end{array}\right.
\label{lame7}
\end{equation}
Once we known the eigenvalues of \eqref{lame7} explicitly, {\bf (P1)-(P2)} are promptly obtained (see \cite{depas} and \cite{ince}).

In view of the expression of the solution $\phi_k$, properties {\bf (P3)-(P4)} are a little bit hard to be obtained. However,  after some algebraic computations one can check that they still hold here. In conclusion, Theorem \ref{regtheorem} can be applied to obtain the orbital stability of $\phi_k$ by the flow of \eqref{bbm}. We refer the reader to \cite{depas} for the details.

\section*{Acknowledgements}

Part of this work was developed during the Ph.D Thesis of the first author, concluded at IMECC-UNICAMP under the guidance of the second author. The first author acknowledges the financial support from Capes and CNPq. The second author is partially supported by CNPq.


\begin{thebibliography}{100}

\bibitem{alb} J.P. Albert, Positivity properties  and  stability of solitary-wave solutions of model equations for long waves,  \textit{Comm. Partial Differential Equations} 17 (1992), 1--22.

\bibitem{albbo} J.P. Albert and J.L. Bona, Total positivity  and the stability of internal waves in fluids of finite depth,  \textit{IMA J. Appl. Math.} 46 (1991), 1--19.

\bibitem{alejo1} 
 M. Alejo, Well-posedness and stability results for the Gardner equation, \textit{NoDEA Nonlinear Differential Equations Appl.} 19 (2012), 503--520.
 
 \bibitem{alejo2}
M. Alejo, C. Mu\~noz, L. Vega, The Gardner equation and the $L^2$-stability of the $N$-soliton solution of the Korteweg-de Vries equation, \textit{Trans. Amer. Math. Soc.} 365 (2013), 195--212.


\bibitem{AN2} J. Angulo and F. Natali, Positivity properties of the Fourier transform and the stability of periodic traveling-wave solutions,  \textit{SIAM J. Math. Anal.} 40 (2008),
1123--1151.

\bibitem{angulo4}
J. Angulo, Nonlinear dispersive equations: Existence and stability of solitary and periodic travelling wave solutions, Math. Surveys Monogr.   156, American Mathematical Society, 2009.

\bibitem{angulo5}
J. Angulo and A. Pastor, Stability of periodic optical solitons for a nonlinear Schr\"odinger system, \textit{Proc. Roy. Soc. Edinburgh Sect. A} 139 (2009), 927--959.





\bibitem{angulo3} J. Angulo, Nonlinear
stability of periodic travelling wave solutions to the
Schr\"odinger and modified Korteweg-de Vries equations,   \textit{J. Differential Equations 235} (2007), 1--30.


\bibitem{acn}
J. Angulo, E. Cardoso and F. Natali, Stability properties of periodic traveling waves for the Intermediate Long Wave equation, \textit{Rev. Mat. Iberoam.} 33 (2017), 417--448.

\bibitem{abs} J. Angulo, C. Banquet, and M. Scialom, Stability for the modified and fourth-order Benjamin-Bona-Mahony equations,   \textit{Discrete Contin. Dyn. Syst.} 30 (2011), 851--871.


\bibitem{angulo1} J. Angulo, J L. Bona, and M. Scialom, Stability of cnoidal waves,  \textit{Adv.  Differential Equations} 11 (2006), 1321--1374.

\bibitem{be}  T.B. Benjamin, The stability of solitary waves,  \textit{Proc. Roy. Soc. London Ser. A} 328 (1972), 153--183.

\bibitem{be1}  T.B. Benjamin, Lectures on nonlinear wave motion, Nonlinear Wave Motion, A. C. Newell, ed., AMS, Providence, R. I. 15 (1974), 3--47.

\bibitem{bbm} T.B. Benjamin, J.L. Bona and J.J. Mahony, Model equations for long waves in nonlinear dispersive systems,  \textit{Phil. Trans. Royal Soc. London Ser. A} 272 (1972), 47--78.

\bibitem{bona} J.L. Bona, On the stability theory of solitary waves, \textit{Proc. R. Soc. Lond. Ser. A} 344 (1975), 363--374.

\bibitem{bss}  J.L. Bona, P.E. Souganidis, and W.A. Strauss, Stability and instability of solitary waves of the Korteweg-de Vries type,  \textit{Proc. R. Soc. Lond. Ser. A} 411 (1987), 395--412.


\bibitem{boka}  N. Bottman and B. Deconinck, KdV cnoidal waves are spectrally stable,  \textit{Discrete Contin. Dyn. Syst.} 25 (2009), 1163--1180.


\bibitem{bjk}  J.C. Bronski, M.A. Johnson, and T. Kapitula, An index theorem for the stability of periodic travelling waves of Korteweg-de Vries type,  \textit{Proc. Roy. Soc. Edinburgh Sect. A} 141 (2011), 1141--1173.

\bibitem{cnp} E. Cardoso Jr., F. Natali, and A. Pastor, Well-posedness and orbital stability of periodic traveling waves for Schamel's equation, \textit{ Z. Anal. Anwend.} 37 (2018), 221--250.

\bibitem{friedman}
P.F. Byrd and M.D. Friedman,  Handbook of elliptic integrals for engineers and scientists, 2nd ed., Springer, NY, 1971.


\bibitem{depas} T.P. de Andrade and A. Pastor, Orbital stability of periodic traveling-wave solutions for the regularized Schamel equation,  \textit{Phys. D} 317 (2016),  43--58.

\bibitem{deka} B. Deconinck and T. Kapitula, The orbital stability of the cnoidal waves of the Korteweg-de Vries,  \textit{Phys. Lett. A} 374 (2010),  4018--4022.




\bibitem{ea} M.S.P. Eastham, The Spectral Theory of Periodic Differential
Equations, Scottish Academic Press, Edinburgh, 1973.





\bibitem{Grillakis} M. Grillakis, M. Shatah and W. Strauss, Stability theory
of solitary waves in the presence of symmetry I,  \textit{J. Funct.
Anal.} 74 (1987), 160--197.

\bibitem{hik} S. Hakkaev, I.D. Iliev, and K. Kirchev,  Stability of periodic traveling waves for complex modified Korteweg-de Vries equation, \textit{J. Differential Equations} 248 (2010), 2608--2627.




\bibitem{ince} E.L. Ince,  The periodic Lam{\'e} functions,  \textit{Proc. Roy. Soc. Edinburgh} 60 (1940), 47--63.

\bibitem{johnson1} M.A. Johnson, Nonlinear stability of periodic traveling wave solutions of the generalized Korteweg-de Vries equation,  \textit{SIAM J. Math. Anal.} 41 (2009), 1921--1947.

\bibitem{johnson2} M.A. Johnson, On the stability of periodic solutions of the generalized Benjamin-Bona-Mahony equation,  \textit{Phys. D} 239 (2010), 1892--1908.

\bibitem{kade} T. Kapitula and B. Deconinck, On the spectral and orbital stability of spatially periodic stationary solutions of generalized Korteweg-de Vries equations,  in {\it Hamiltonian Partial Differential Equations and Applications}, Vol. 75, Fields Institute Communications, 285--322, 2015.

\bibitem{kar}
S. Karlin, Total Positivity, Stanford University Press, Stanford, CA, 1968.

\bibitem{kkb} D R.S. Ko, T. Kubota, and L.D. Dobbs, Weakly-nonlinear, long internal gravity waves in stratified fluids of finite depth, \textit{J. Hydronautics} 12 (1978), 157--165. 

\bibitem{konno} K. Konno, Y.H. Ichikawa,  A modified Korteweg de Vries equation for ion acoustic waves, \textit{J. Phys. Soc. Jpn.} 37 (1974), 1631--1636.

\bibitem{kdv} D.J. Korteweg and G. de Vries, On the change of form of long waves advancing in a rectangular canal, and on a new type of long stationary waves,   \textit{Philos. Mag.} 39   (1895), 422--443.

\bibitem{kudryashov}
N.A. Kudryashov, D.I. Sinelshchikov, A note on the Lie symmetry analysis and exact solutions for the extended mKdV equation, \textit{Acta Appl. Math.}  113 (2011), 41--44.



\bibitem{Magnus} W. Magnus and S. Winkler,   Hill's equation,
Interscience, Tracts in Pure and Appl. Math., vol  20, 1976.

\bibitem{natali2} F. Natali and A. Neves, Orbital stability of periodic waves,  \textit{IMA J. Appl. Math.} 79   (2014), 1161-1179.

\bibitem{NP1} F. Natali and A. Pastor,
Stability properties of periodic standing waves for the
Klein-Gordon-Schrödinger system, \textit{Commun. Pure Appl. Anal.} 9 (2010), 413--430.

\bibitem{natali-pastor} F. Natali and A. Pastor, Stability and instability of periodic standing wave solutions for some Klein-Gordon equations, \textit{J. Math. Anal. Appl.} 347 (2008), 428--441.


\bibitem{Neves1}
A. Neves,  Floquet's Theorem and stability of periodic solitary waves, \textit{J. Dynam. Differential Equations} 21 (2009), 555--565.

\bibitem{Neves2}
A. Neves, Isoinertial family of operators and convergence of KdV cnoidal waves to solitons, \textit{J. Differential Equations} 244 (2008), 875--886.

\bibitem{sch}
H. Schamel, Stationary solitary, Snoidal and Sinusoidal ion acoustic waves, \textit{Plasma Physics} 14 (1972), 905--924.

\bibitem{sch1}
H. Schamel, A modified Korteweg-de Vries equation for ion acoustic waves due to resonant electrons, \textit{J. Plasma Phys.} 9 (1973), 377--387.



\bibitem{wadati1} M. Wadati, Wave propagation in nonlinear lattice I, \textit{J. Phys. Soc. Jpn.} 38 (1975), 673--680.


\bibitem{wadati2} M. Wadati, Wave propagation in nonlinear lattice II. \textit{J. Phys. Soc. Jpn.} 38 (1975), 681--686.

\bibitem{weinstein1} M.I. Weinstein, Lyapunov stability of ground states of nonlinear dispersive evolution equations,  \textit{Comm. Pure Appl. Math.}  39 (1986),  51--67.



 \end{thebibliography}
\end{document}